\documentclass[a4paper,11pt,oneside]{article}
\usepackage{amsfonts,amsmath,amsthm,amssymb,verbatim,pstricks-add,pst-arrow,xcolor}
\usepackage{hyperref}
\usepackage{fullpage}
\usepackage[normalem]{ulem}
\usepackage{tikz}
\usepackage{url}
\usepackage{etoolbox}

\newcommand{\Z}{\mathbb{Z}}
\renewcommand{\O}{\mathcal{O}}

\newcommand{\ZZ}{\mathfrak{Z}}

\def\B{\mathbf{B}}
\def\C{\mathbb{C}}

\def\S{\Sigma}
\def\a{\alpha}
\def\b{\beta}

\DeclareMathOperator{\SF}{SF}

\def\QQ{\mathbb{Q}}

\newtheorem{theorem}{Theorem}
\newtheorem{lemma}[theorem]{Lemma}
\newtheorem{corollary}[theorem]{Corollary}
\newtheorem{proposition}[theorem]{Proposition}

\theoremstyle{remark}
\newtheorem{remark}[theorem]{Remark}

\theoremstyle{definition}
\newtheorem{example}[theorem]{Example}

\numberwithin{theorem}{section}

\AtBeginEnvironment{example}{%
  \pushQED{\qed}}
\AtEndEnvironment{example}{\popQED}
\AtBeginEnvironment{remark}{%
  \pushQED{\qed}}
\AtEndEnvironment{remark}{\popQED}
\AtBeginEnvironment{proof}{}

\allowdisplaybreaks

\title{Invariant factors from products of cyclotomic polynomials,\\ with applications}

\author{Vanni Noferini\thanks{Corresponding author. Supported by a Research Council of Finland grant (decision number 370932).
Department of Mathematics and Systems Analysis, Aalto University, P.O. Box 11100, FI-00076, Aalto, Finland.
E-mail: \href{mailto:vanni.noferini@aalto.fi}{vanni.noferini@aalto.fi}}\quad and Gerald Williams\thanks{Partially supported by a London Mathematical Society Scheme 4 (Research in Pairs) grant 42518. School of Mathematics, Statistics and Actuarial Science, University of Essex, Colchester, Essex CO4 3SQ, United Kingdom. E-mail: \href{mailto:gerald.williams@essex.ac.uk}{gerald.williams@essex.ac.uk}}}

\begin{document}
\maketitle

\begin{abstract}
Let $f,g\in\Z[t]$ be products of cyclotomic polynomials such that either $g/\gcd(f,g)$ or $f/\gcd(f,g)$ (or both) is squarefree. We give explicit formulae for the invariant factors of $f(C_g)$, where $C_g$ is the companion matrix of $g$, expressing them in terms of the indices and multiplicities of the cyclotomic factors of $f$ and $g$. A more compact formula holds for the circulant case $g(t)=t^n-1$, assuming that every cyclotomic factor shared by $f$ and $g$ has multiplicity $1$ in $f$. We discuss applications to several topics, including abelianizations of cyclically presented groups, homology groups of cyclic branched covers of knots, adjacency matrices and shifted graph Laplacians, $K$-theory, and Smith forms of matrix powers. For example, for two families of hyperbolic knots we give the first integral homology in every covering degree; or, when $g$ is squarefree, we  determine the least eventual period of the Smith forms of powers of $f(C_g)$.
\end{abstract}

{\small
\par}

\section{Introduction}

In \cite{NoferiniWilliamsBrieskorn,NoferiniWilliamsApplications},
Noferini and Williams study Smith forms of matrices in the \emph{companion ring} of a monic polynomial. We recall this setting over an arbitrary Elementary Divisor Domain (EDD) $R$
\cite{Friedland2016}, which in this paper will typically (but not always) be chosen as $R=\Z$. Let $g\in R[t]$ be monic of degree $n\geq1$, written as $g(t)=t^n+\sum_{j=0}^{n-1}g_jt^j$. Following \cite[Section 3]{NoferiniWilliamsBrieskorn}, we define the companion matrix of $g$ as
\[ C_g:=\begin{bmatrix} -g_{n-1}&-g_{n-2}&\cdots&-g_0\\ 1&0&\cdots&0\\ &\ddots&\ddots&\vdots\\ 0&&1&0
 \end{bmatrix}\in R^{n\times n}.\]
 Occasionally, we also allow $g=1$; in this case, $C_g \in R^{0 \times 0}$ is the empty matrix, and we conventionally agree that its determinant is $1$. The companion ring of $g$ is
\[ R[C_g]:=\{f(C_g):f\in R[t]\}\cong R[t]/\langle g\rangle,\]
where the isomorphism sends the residue class $[f]$ to $f(C_g)$ \cite{NoferiniWilliamsBrieskorn,NoferiniWilliamsApplications}. We write the monomial basis as a column, ordered by decreasing degree, and vectors of polynomial coefficients as rows. With this convention,
\[ C_g\begin{bmatrix}t^{n-1}\\\vdots\\t\\1\end{bmatrix} \equiv t\begin{bmatrix}t^{n-1}\\\vdots\\t\\1\end{bmatrix} \bmod g.\]
Thus, multiplication by $f$ modulo $g$ sends a coefficient row $a^T$ to $a^T f(C_g)$; see
\cite[Section 2.3]{NoferiniWilliamsApplications}.

When $g(t)=t^n-1$, the matrix $C_g$ is a cyclic permutation matrix, and its companion ring is the ring of circulant matrices over $R$; see \cite{Davis1979}. One reason to study the Smith forms of integer circulants comes from cyclically presented groups.  These are groups defined by a presentation with generators $x_0,\ldots,x_{n-1}$ and whose defining relators are obtained by
cyclically shifting the subscripts in a single word $w$. If $a_j$ is the exponent sum of $x_j$ in $w$, the polynomial $f(t):=\sum_{j=0}^{n-1}a_jt^j$ is called the \emph{representer polynomial}. The circulant $f(C_{t^n-1})$ presents the abelianization of the group. Its Smith form therefore determines that abelianization and, when the group is the fundamental group of a 3-dimensional manifold, the first integral homology of the manifold; see \cite[Section 2.2]{NoferiniWilliamsBrieskorn} and \cite[Section 2.4]{NoferiniWilliamsApplications}.

In this paper, we focus on the case where $f,g\in\Z[t]$ are products of cyclotomic polynomials and either $F:=f/\gcd(f,g)$ or $G:=g/\gcd(f,g)$ is squarefree; as we will see below, we may assume without loss of generality that $G$ is squarefree. This assumption is in particular always satisfied if $g$ itself is squarefree, which happens naturally in most applications that we are aware of; for example, when $g(t)=t^n-1$ so that $f(C_g)$ is a circulant matrix. A squarefree $g$ corresponds to the matrix $C_g$ having finite multiplicative order, i.e., there exists a positive integer $N>0$ such that $C_g^N=I$. Our goal is to express the invariant factors of $f(C_g)$ explicitly in terms of the indices and the multiplicities of the cyclotomic factors of $f$ and $g$. A related study, that analyzed the Smith form over $\QQ[\lambda]$ of circulant polynomial matrices, was carried out in \cite{TelloniWilliams}. We denote by $\Phi_d$ the $d$-th cyclotomic polynomial, whose roots are the primitive $d$-th roots of unity. Its degree is Euler's totient function $\varphi(d):=\#\{1\leq j\leq d:\gcd(j,d)=1\}$; here and throughout, $\#S$ denotes the cardinality of a finite set $S$.

Given a prime $p$ and a positive integer $n$, we write $\beta:=\nu_p(n)$ and $\alpha:=n/p^\beta$, so that $n=\alpha p^\beta$ and $p\nmid\alpha$. Here $\nu_p(n)$, the \emph{$p$-adic valuation} of $n$, is the exponent of $p$ in its prime factorization \cite[Chapter II, Section 3]{Neukirch99}. The integer $\alpha$ is the \emph{coprime-to-$p$ part} of $n$.

A square integer matrix $U$ is \emph{unimodular} if $\lvert\det U\rvert=1$; more generally, a square matrix over a ring is unimodular if it is invertible over that ring. We write $M\sim N$ to mean that $M$ and $N$ are unimodularly equivalent, that is, $N=UMV$ for unimodular matrices $U,V$. The Smith Theorem, first proved over $\Z$ by Smith \cite{Smith1861}, states that every integer matrix $M\in\Z^{n\times n}$ of rank $r$ is unimodularly equivalent to a unique matrix
\[ \SF(M):=\left(\bigoplus_{i=1}^r s_i\right) \oplus0_{(n-r)\times(n-r)}, \qquad s_i>0,\quad s_i\mid s_{i+1}\quad(1\leq i<r).\]
We call $\SF(M)$ the \emph{Smith form} of $M$ and its positive diagonal entries the \emph{invariant factors}. The prime powers $p^{\nu_p(s_i)}$ with $\nu_p(s_i)>0$ are the \emph{elementary divisors}; we call the exponents $\nu_p(s_i)$ the \emph{partial multiplicities at $p$} \footnote{There is no universal agreement in the literature on whether the zero diagonal elements should be included among the invariant factors or not; in this paper, we do not include them. Note, however, that our definition of partial multiplicities does instead include zeros (corresponding to unit invariant factors).}. More generally, the Smith Theorem holds over any elementary divisor domain \cite[Theorem 1.14.1]{Friedland2016}, and in particular over any PID. In this extended setting, the invariant factors are only unique up to multiplication by units, and hence the Smith form can be made unique by picking a convention analogous to choosing nonnegative integers for the invariant factors over $\Z$. We will often work over the ring of integers of a field extension of $\QQ_p$, and pick powers of $p$ for the invariant factors.

The Smith Theorem is closely related to the structure theorem for finitely presented modules over an EDD (or finitely generated modules over a PID). Over $\Z$, we have
\[ \operatorname{coker}M:=\Z^{1\times n}/(\Z^{1\times n}M) \cong\bigoplus_{i=1}^r\Z/s_i\Z\oplus\Z^{n-r}.\]
For example, let $p$ be prime and $m\geq1$, and write $m-1=:q(p-1)+r$, $0\leq r<p-1$. Then, as we will see in Example \ref{ex:example1}, 
\begin{equation}\label{eq:example1}
 \SF\left((C_{t^p-1}-I)^m\right) =p^q I_{p-1-r}\oplus p^{q+1}I_r\oplus0.
\end{equation}
 This gives the decomposition
\[ \operatorname{coker}(C_{t^p-1}-I)^m \cong(\Z/p^q\Z)^{p-1-r}  \oplus(\Z/p^{q+1}\Z)^r\oplus\Z.\]
The same example also showcases a connection with the $K$-theory of the lens space  obtained by identifying points of the unit sphere in $\C^m$ that differ by multiplication by a $p$-th root of unity.  We explain this link in Example \ref{ex:lensgroups}, and we treat more general cyclic actions on spheres in  Section \ref{sec:Ktheory}.

Let us briefly recall some results on matrices in companion rings that can already be found in the literature. The most relevant for the present paper is \cite[Theorem C]{NoferiniWilliamsBrieskorn}, which, for coprime integers $r,s\geq2$ and $n\geq2$, yields the Smith form of $f(C_g)$, where
\[ f(t):=\frac{(t^{rs}-1)(t-1)}{(t^r-1)(t^s-1)},\qquad g(t):=t^n-1.\]
Here $f$ is the Alexander polynomial of the torus knot $T(r,s)$, and is a squarefree product of cyclotomic polynomials. We will give several generalizations of \cite[Theorem C]{NoferiniWilliamsBrieskorn}; for example,  Corollary \ref{thm:originale} extends it to arbitrary squarefree products $f$, while the more general  Theorem \ref{thm:global} also allows repeated cyclotomic factors in $f$ and a broader class of choices for $g$. In \cite{Apostol70}, Apostol determines the resultants of cyclotomic polynomials. By the identity $\lvert\det f(C_g)\rvert=\lvert\operatorname{Res}(f,g)\rvert$, this equivalently gives the product of the invariant factors whenever $f$ and $g$ are coprime. For a single pair of cyclotomic polynomials, the complete Smith form of $\Phi_a(C_{\Phi_b})$ is given in \cite[Theorem 7.4]{NoferiniWilliamsBrieskorn}. We will frequently use three further results from
\cite{NoferiniWilliamsBrieskorn,NoferiniWilliamsApplications}:
\begin{itemize}
    \item First, if $z:=\gcd(f,g)$ is the monic greatest common divisor of $f$ and $g$, then
\cite[Theorem A]{NoferiniWilliamsBrieskorn} states
\[ f(C_g)\sim (f/z)(C_{g/z})\oplus0_{\deg z\times\deg z}.\]
This determines the number of zero diagonal entries and reduces the problem to coprime polynomials. 
\item Second, the \emph{Swap Theorem} \cite[Theorem 3.11]{NoferiniWilliamsApplications} states that, for monic $f,g$ with $m:=\deg f\leq n:=\deg g$,
\[ f(C_g)\sim I_{n-m}\oplus g(C_f).\]
We will use it when interchanging the two polynomials gives a smaller, or simpler, matrix. Together with Theorem \ref{thm:global}, the Swap Theorem also gives a formula when $f/\gcd(f,g)$ is squarefree. 
\item Third, the \emph{Composition Theorem} \cite[Theorem 3.13]{NoferiniWilliamsApplications} states that, for every nonconstant monic polynomial $h$, the matrix $(f \circ h)(C_{g \circ h})$ is unimodularly similar to the direct sum of $\deg h$ copies of $f(C_g)$. 
\end{itemize}
Moreover, \cite[Theorems 6.1 and 6.3]{NoferiniWilliamsBrieskorn} describe how factorizations of $f$ and $g$ can be used, under appropriate assumptions, to determine the Smith form. If $f=f_1f_2$ and $\operatorname{Res}(f_1,g)$ and $\operatorname{Res}(f_2,g)$ are coprime, then
\[\SF(f(C_g))=\SF(f_1(C_g))\,\SF(f_2(C_g));\]
and if $g=g_1g_2$, where $g_1,g_2$ are monic, and $\operatorname{Res}(f,g_1)$ and $\operatorname{Res}(f,g_2)$ are coprime, then
\[f(C_g)\sim f(C_{g_1})\oplus f(C_{g_2}).\]
In the setting of the present manuscript, we still wish to exploit a factorization of $g$, but the resultants involved in such factorization may have common prime factors and so we need to develop new techniques to handle this situation. In particular, in Theorem \ref{lem:bello} we exploit a factorization by assuming that $\operatorname{Res}(g_1,g_2)$ is a unit (no assumptions on $f$ are needed). Theorem \ref{thm:collage} also achieves a block diagonalization, specific to the case where $f,g$ are products of cyclotomic polynomials and under more technical assumptions. 

Our proofs rely on some basic number theory; background on the concepts used here can be found in \cite[Chapter II, Sections 3--7]{Neukirch99}. Let $\QQ_p$ denote the field of $p$-adic numbers and let $\Z_p$ be its ring of integers; we also use $\mathbb F_p:=\Z/p\Z$ for the finite field with $p$ elements. The valuation $\nu_p$ extends to $\nu_p:\QQ_p\longrightarrow\Z\cup\{+\infty\}$, normalized by $\nu_p(p)=1$ and $\nu_p(0):=+\infty$. When studying local Smith forms over $\Z_p$, we normalize nonzero invariant factors to be powers of $p$.

We present our main result in two steps. Theorem \ref{thm:closed} provides the local Smith form over $\Z_p$ when the cyclotomic indices are powers of $p$. In Theorem \ref{thm:global}, we reduce arbitrary cyclotomic indices to this case and combine the local Smith forms using the local-global principle of Lemma \ref{lem:reordersmith}. Corollary \ref{thm:originale} gives a more compact formula for
integer circulant matrices when all the shared cyclotomic factors have multiplicity $1$ in $f$. The proofs of the main results are then given in Section \ref{sec:proof}.
We turn to applications in Section \ref{sec:applications}, which begins by considering the so-called cyclically presented groups of type $\ZZ$ and calculates the abelianization of a subfamily of these groups. For knots of tunnel number at most $1$ in the 3-sphere $S^3$ whose Alexander polynomials are products of cyclotomic polynomials, we study the first integral homology of cyclic branched covers. In particular, Corollaries \ref{cor:Km} and \ref{cor:pretzel} give this homology in every covering degree for two families of hyperbolic knots. We also treat a family of $(1,1)$-knots in lens spaces, and we discuss cable knots. Turning to graph theory, we determine Smith forms of adjacency matrices of powers of cycles and shifted Laplacians of certain Cartesian products of graphs. In Section \ref{sec:Ktheory}, we apply Theorem \ref{thm:global} to equivariant $K$-theory. Corollary \ref{cor:illustKtheory} gives the additive $K$-groups for a family of nonfree actions of cyclic groups of order $p^3$ on spheres, and Example \ref{ex:lensgroups} recovers classic $K$-groups of lens spaces. Section \ref{sec:matrixpowers} connects these $K$-groups with the eventual periodicity of Smith forms of integer matrix powers \cite{Nof26}. Theorem \ref{thm:final} gives, when $g$ is squarefree, the least eventual period of the Smith forms of powers of $f(C_g)$ as the least common denominator of all $p$-adic valuations of its nonzero eigenvalues.

\section{Main results}

Let us start with the simple Lemma \ref{lem:reordersmith}, which records both how to obtain the Smith form of a diagonal matrix and how to reconstruct integer invariant factors from their
partial multiplicities at each prime. Both results are folklore; the latter is known as the local-global principle, and for a reordering argument over a polynomial ring that is related to the former see, for example, \cite[Proof of Theorem 4.1, Section 5.1]{Noferini12}.
\begin{lemma}\label{lem:reordersmith}
Let $a_1,\ldots,a_r$ be positive integers. For each prime $p$, let $\eta_{p,1}\leq\cdots\leq \eta_{p,r}$ be the list of the $p$-adic valuations of the indexed list $(a_i)_{i=1}^r$, rearranged in increasing order. Then
\[ \SF\left(\bigoplus_{i=1}^r a_i\right) =\bigoplus_{i=1}^r s_i, \qquad s_i:=\prod_p p^{\eta_{p,i}}.\]
More generally, if $A$ is an integer matrix having rank $r$ and nonzero local invariant factors $p^{\eta_{p,1}},\ldots,p^{\eta_{p,r}}$ over $\Z_p$, then its invariant factors over $ \Z $ are given by $s_i = \prod_p p^{\eta_{p,i}}$, $i=1,\dots,r$.
\end{lemma}

\begin{proof}
It is clear that the minors of order $k$ of a diagonal matrix are either $0$ or the products of $k$ diagonal entries (or both, if the matrix is singular). Hence, the $p$-adic valuation of their gcd is
$\eta_{p,1}+\cdots+\eta_{p,k}$. Since this gcd is the product of the first $k$ invariant factors, by taking the quotient of the products for $k$ and $k-1$ we get $\nu_p(s_k)=\eta_{p,k}$. Repeating this at every prime and using the fact that the integers are a unique factorization domain yields the first part of the statement.

For the second part of the statement, it suffices to note that unimodular equivalence over
$\Z_p$ preserves the determinantal ideals, i.e., the ideals of $\Z_p$ generated by the minors of each
order. We may therefore replace $A$ by its Smith form over $\Z$ and apply the first part to its nonzero diagonal entries.
\end{proof}

By \cite[Theorem A]{NoferiniWilliamsBrieskorn}, we may always divide by the gcd of $f$ and $g$, thus reducing to the case of coprime polynomials $F$ and $G$. We assume that $G$ is squarefree. Our next step is to obtain the Smith form of $F(C_G)$ over $\Z_p$ when all cyclotomic indices occurring in either $F$ or $G$ are powers of a fixed prime $p$. We will later explain how to combine these local formulae to obtain the Smith form over $\Z$ for arbitrary cyclotomic indices. In this section, we focus on statements and remarks; all but the easiest proofs are postponed to Section \ref{sec:proof}.

\subsection{Local Smith forms in the case of prime power indices}

Fix a prime $p$. The goal of this subsection is to state Theorem \ref{thm:closed}, which gives the Smith form of $F(C_G)$ over $\Z_p$ under the assumptions that (i) $F$ and $G$ are coprime products of cyclotomics, with $G$ squarefree, and (ii) the indices of the cyclotomic factors are all powers of $p$. While its formal statement is somewhat technical, the essence of Theorem \ref{thm:closed} is that to each factor $\Phi_{p^b}$ of $G$ is associated a diagonal block having size $\varphi(p^b)$ (that is, $p^b-p^{b-1}$ for $b>0$ or $1$ for $b=0$); the elements of this diagonal block depend (in a highly nonlinear way that is nonetheless expressible by closed formulae plus one Euclidean division) both on which indices of cyclotomic factors in $F$ and $G$ are smaller or are larger than $p^b$, and on the multiplicities of these factors in $F$. Taking the direct sum of these blocks yields a diagonal matrix unimodularly equivalent to $F(C_G)$. This diagonal matrix may not be directly the Smith form, but if not then one can easily obtain the latter by invoking Lemma \ref{lem:reordersmith} and reordering.

\begin{theorem}\label{thm:closed}
Let $p$ be a prime, and let $A,B$ be disjoint finite sets of nonnegative integers, with $B\neq\varnothing$. For each $h\in A$, let $k_h$ be a positive integer, and define
\[ F(t):=\prod_{h\in A}\Phi_{p^h}(t)^{k_h},\qquad G(t):=\prod_{b\in B}\Phi_{p^b}(t). \]
For every $b\in B$, define $q_b,r_b$ by the Euclidean division
\begin{equation}\label{eq:division}
 \max\left(0,\sum_{\substack{h\in A\\h<b}}\varphi(p^h)k_h-
                   \sum_{\substack{c\in B\\c<b}}\varphi(p^c)\right)
       =:q_b\varphi(p^b)+r_b,\qquad 0\leq r_b<\varphi(p^b).
\end{equation}
Also define
\begin{equation}\label{eq:gap}
 \alpha_b:=\sum_{\substack{h\in A\\h>b}}k_h+\#\{c\in B:c>b\},\quad
 \delta_b:=\#\{c\in B:c>b,\
       \nexists \ h\in A\text{ with }b<h<c\}.
\end{equation}
Then, over $\Z_p$,
\begin{equation}\label{eq:closed}
 F(C_G)\sim\bigoplus_{b\in B}
 \begin{cases}
  p^{\alpha_b-\delta_b}I_{\varphi(p^b)-r_b}
       \oplus p^{\alpha_b+1}I_{r_b},&\text{if }q_b=0,\\p^{\alpha_b+q_b}I_{\varphi(p^b)-r_b} \oplus p^{\alpha_b+q_b+1}I_{r_b},&\text{if }q_b\geq1.
 \end{cases}
\end{equation}
\end{theorem}

Let us stop for a moment to interpret the exact meaning of the ingredients of Theorem \ref{thm:closed}. The left hand side of \eqref{eq:division} is the difference between the degrees of the factors of $F$ and $G$ consisting of those cyclotomic factors of indices less than $p^b$ (or zero if such difference is negative). The quantities $q_b$ and $r_b$ are then obtained as the quotient and remainder after division by $\varphi(p^b)$. The integer $\alpha_b$ in \eqref{eq:gap} counts, with multiplicity, the factors in $F$ and $G$ having index larger than $p^b$. To interpret $\delta_b$, the other integer defined in \eqref{eq:gap}, suppose first that $F$ has a cyclotomic factor of index higher than $p^b$, and call the smallest such index $p^h$. Then, $\delta_b$ counts the cyclotomic factors of $G$ having index strictly between $p^b$ and $p^h$. If, instead, all the factors of $F$ have index less than $p^b$, then $\delta_b$ counts all factors of $G$ with index greater than $p^b$. 
\begin{example}
Take $F(t)=\Phi_2(t)^2$, $G(t)=\Phi_1(t)\Phi_4(t)\Phi_8(t)$, and $p=2$. Therefore, $A=\{1\}$, $B=\{0,2,3\}$, and $k_1=2$. Applying the prescriptions of Theorem \ref{thm:closed},
we readily get the following table.
\[ \begin{array}{c|c|c|c|c|c}
 b&\varphi(2^b)&(q_b,r_b)&\alpha_b&\delta_b&
 \text{invariant factors}\\ \hline
 0&1&(0,0)&4&0&2^4\\
 2&2&(0,1)&1&1&1,\ 2^2\\
 3&4&(0,0)&0&0&1\text{ ($4$ times)}
 \end{array}\]
Rearranging the invariant factors, we see that the Smith form of $F(C_G)$ over $\Z_2$ is $I_5\oplus4\oplus16$.
\end{example}
From \eqref{eq:closed}, we see that at most $2$ distinct partial multiplicities in the Smith form of $F(C_G)$ are associated with each cyclotomic factor of $G$. More precisely, a factor contributes $1$ distinct partial multiplicity if and only if $r_b = 0$, and $2$ distinct partial multiplicities otherwise. When $r_b > 0$, those $2$ distinct partial multiplicities differ by $1$ when
$q_b\geq1$, and by $\delta_b+1$ when $q_b=0$.

\subsection{Global Smith forms for arbitrary indices}

Theorem \ref{thm:global} is our main result, and generalizes Theorem \ref{thm:closed} to arbitrary indices of the cyclotomic factors. Its proof will be given in Section \ref{sec:proof}.

\begin{theorem}\label{thm:global}
Let $A,B$ be finite sets of positive integers, with $B\neq\varnothing$, and let $u(d)$, for $d\in A$, and $v(d)$, for $d\in B$, be positive integers. Let
\[ f(t):=\prod_{d\in A}\Phi_d(t)^{u(d)},\qquad g(t):=\prod_{d\in B}\Phi_d(t)^{v(d)},\]
 $z:=\gcd(f,g)$, $F:=f/z$, $G:=g/z$, $\ell:=\deg z$, and $m:=\deg G$. Suppose that $G$ is squarefree.
For every $d\in A$, define
\[ k(d):=\begin{cases}  \max\{u(d)-v(d),0\},&\text{if }d\in B,\\  u(d),&\text{if }d\notin B,\end{cases}\]
so that $F(t)=\prod_{d\in A}\Phi_d(t)^{k(d)}$. For every prime $p$ and every positive integer $a$ coprime to $p$, define
\[ A_{p,a}:=\{h\in\Z_{\geq0}:ap^h\in A,\ k(ap^h)>0\},\qquad  B_{p,a}:=\{h\in\Z_{\geq0}:\Phi_{ap^h}\mid G\}. \]
For each such pair $(p,a)$ with $B_{p,a}\ne\varnothing$, let $D_{p,a}$ be the diagonal matrix on the right hand side of \eqref{eq:closed} for the sets $A_{p,a}$, $B_{p,a}$, and multiplicities $k_h=k(ap^h)$. Then, for every prime $p$,
\begin{equation}\label{eq:globallocal}
 f(C_g)\sim \left(\bigoplus_{\substack{a\geq1,\ p\nmid a\\B_{p,a}\ne\varnothing}} \bigoplus_{j=1}^{\varphi(a)}D_{p,a}\right)  \oplus0_{\ell\times\ell} \qquad \text{over} \ \Z_p.
\end{equation}
Let $\eta_{p,1}\leq\cdots\leq\eta_{p,m}$ be the $p$-adic valuations of the nonzero diagonal entries on the right hand side of \eqref{eq:globallocal}, arranged in increasing order. Then
\begin{equation}\label{eq:global}
 \SF(f(C_g))=\left(\bigoplus_{i=1}^m d_i\right)\oplus0_{\ell\times\ell},
 \qquad d_i:=\prod_p p^{\eta_{p,i}}.\end{equation}
\end{theorem}

For every prime $p$, the sets $A_{p,a}$ and $B_{p,a}$ contain the exponents $h$ in the indices $ap^h$ of the factors of $F$ and $G$, respectively. In \eqref{eq:globallocal}, every such pair with $B_{p,a}\ne\varnothing$ is associated with $ \varphi(a)$ diagonal blocks $D_{p,a}$ of the form described by Theorem \ref{thm:closed}. In addition to all these blocks, there are $\ell$ zero diagonal entries associated with the common factors of $f$ and $g$. To obtain the integer invariant factors in \eqref{eq:global}, we follow Lemma \ref{lem:reordersmith}, first reordering the partial multiplicities at each prime and then multiplying together the prime powers occupying the same position.

\begin{example}\label{ex:parameters}
Let $u,v\geq1$, and define 
\[ f(t)=\Phi_1(t)\Phi_6(t)^u\Phi_{24}(t)^v, \qquad g(t)=\Phi_3(t)\Phi_{12}(t)\Phi_{48}(t). \]
These polynomials are coprime, so $F=f$ and $G=g$. For $p=2$, all indices of the factors of $G$ have coprime-to-$2$ part $3$. Thus $A_{2,3}=\{1,3\}$ and $B_{2,3}=\{0,2,4\}$. In the notation of Theorem \ref{thm:closed}, we get $\delta_0=\delta_2=\delta_4=0$. The local blocks with $b=0$ and $b=2$, corresponding to the factors $\Phi_3$
and $\Phi_{12}$ of $G$, give, respectively, the partial multiplicities
\[ u+v+2,\qquad v+1+\left\lfloor\frac{u-1}{2}\right\rfloor,\qquad v+1+\left\lfloor\frac u2\right\rfloor. \]
For the block associated with $\Phi_{48}$, where $b=4$, we write $u+4v-3=:8q+r$, with $0\leq r<8$. Its partial multiplicities are $q$ ($8-r$ times) and $q+1$  ($r$ times). According to Theorem \ref{thm:global}, each partial multiplicity obtained from these three blocks must be
repeated $\varphi(3)=2$ times. Rearranging, the Smith form over $\Z_2$ is therefore
\[  D_2:=2^q I_{16-2r}\oplus2^{q+1}I_{2r}\oplus2^{v+1+\lfloor(u-1)/2\rfloor}I_2 \oplus2^{v+1+\lfloor u/2\rfloor}I_2 \oplus2^{u+v+2}I_2. \]
For $p=3$, the only coprime-to-$3$ part shared by indices of factors of $F$ and $G$ is $1$, with $A_{3,1}=\{0\}$ and $B_{3,1}=\{1\}$. Theorem \ref{thm:global} gives $D_3:=I_{21}\oplus3$. By \cite[Theorem 7.4]{NoferiniWilliamsBrieskorn}, no other prime contributes an elementary
divisor: all the cyclotomic indices involved have only $2$ and $3$ as prime divisors. Hence the Smith form over $\Z$ is $D_2D_3$; multiplying by $D_3$ multiplies only the last entry of $D_2$ by $3$. For a concrete example, when $u=5$ and $v=3$ we obtain
\[ \SF(f(C_g))=2I_4\oplus4I_{12}\oplus64I_4\oplus1024\oplus3072. \qedhere \]
\end{example}

\begin{remark}
The assumption that $G$ is squarefree, which in terms of the notation introduced in Theorem \ref{thm:global} is equivalent to $ v(d)=1  $ whenever $d \in B \setminus A$ and $ v(d) - u(d) \leq 1 $ whenever $d \in A \cap B$,  is strictly weaker than requiring $g$ to be squarefree. For example, after cancellation of common factors,
\[ f(t)=\Phi_2(t)^5\Phi_3(t)\Phi_4(t)^2,\qquad g(t)=\Phi_1(t)\Phi_2(t)^2\Phi_5(t)\Phi_4(t)^3\]
reduces to
\[ F(t)=\Phi_2(t)^3\Phi_3(t),\qquad G(t)=\Phi_1(t)\Phi_4(t)\Phi_5(t).
\]
There is also another obvious symmetric extension of Theorem \ref{thm:global}, namely, combining Theorem \ref{thm:global} with the Swap Theorem \cite[Theorem 3.11]{NoferiniWilliamsApplications} yields a formula also for the case where $F$ is squarefree but $G$ is not.
\end{remark}
\subsection{A corollary for circulants}

Motivated by the many applications that involve integer circulant matrices (see Section \ref{sec:applications}), we specialize Theorem \ref{thm:global} to $g(t)=t^n-1$, and we assume that every common cyclotomic factor has multiplicity $1$ in $f$ (in particular this always holds if $f$ is squarefree). We emphasize that this restriction concerns only factors dividing both
$f$ and $g$; the other factors of $f$ may have arbitrary multiplicities. We explicitly state the resulting Smith form, and we give a short proof to show how it follows from Theorem \ref{thm:global}.

\begin{corollary}\label{thm:originale}
Let $n\geq1$, $g(t):=t^n-1$, and $f(t):=\prod_{d\in A}\Phi_d(t)^{u(d)}$, where $A$ is a finite
set of positive integers and $u(d)$ is a positive integer for each $d\in A$. Suppose that every
cyclotomic factor common to $f$ and $g$ has multiplicity $1$ in $f$, that is, $u(d)=1$ for every $d\in A$ dividing $n$. Define $\ell:=\sum_{\substack{d\in A\\d\mid n}}\varphi(d)$. Given a prime $p$, write $n=\alpha p^\beta$, $p\nmid\alpha$. For every $a\mid\alpha$, define
\[ \mu_{p,a}:=\sum_{\substack{0\leq h\leq\beta\\ap^h\notin A}} \varphi(ap^h), \qquad n_{p,a}:=\sum_{\substack{h>\beta\\ap^h\in A}}u(ap^h).\]
Let $a_1,\ldots,a_s \in \Z$ be the divisors of $\alpha$, ordered so that $n_{p,a_1}\leq\cdots\leq n_{p,a_s}$, and define
\[ D_p:=\bigoplus_{j=1}^s p^{n_{p,a_j}}I_{\mu_{p,a_j}}  \in\Z^{(n-\ell)\times(n-\ell)}.\]
Then,
\begin{equation}\label{eq:easySF}
 \SF\left(f(C_{t^n-1})\right) =\left(\prod_p D_p\right)\oplus0_{\ell\times\ell},
\end{equation}
where the product is over all primes dividing $n\prod_{d\in A}d$.
\end{corollary}
\begin{proof}
Specializing the notation of Theorem \ref{thm:global} to $g(t)=t^n-1$, we have $B_{p,a}=\{0\leq h\leq\beta:ap^h\notin A\}$. Fix a prime $p$ and a divisor $a \mid \alpha$ such that $B_{p,a} \neq \varnothing$. By the assumptions, for every factor of $F$ with index $ap^h$ it must hold $h>\beta$, and hence $h \not\in B_{p,a}$. Therefore, for every $b\in B_{p,a}$, the left hand side of \eqref{eq:division} is zero, implying $q_b=r_b=0$. Moreover, since every $h\in A_{p,a}$ satisfies $h>\beta$, $\delta_b=\#\{c\in B_{p,a}:c>b\}$ and $\alpha_b=n_{p,a} + \delta_b$. Substituting these values in \eqref{eq:closed}, we obtain the block $p^{n_{p,a}}I_{\varphi(p^b)}$ for each $b\in B_{p,a}$. In \eqref{eq:globallocal}, there are $\varphi(a)$ copies of every such block. Observe now that $\varphi(a)\sum_{b\in B_{p,a}}\varphi(p^b)=\sum_{b\in B_{p,a}}\varphi(ap^b)=\mu_{p,a}$; therefore, we can combine all the blocks associated with $a$ into $p^{n_{p,a}} I_{\mu_{p,a}}$. We can repeat this argument for every divisor $a = a_1,\ldots,a_s$; since these are ordered so that the corresponding blocks are in increasing divisibility order, we conclude that $D_p$ is the Smith form of $F(C_G)$ over $\Z_p$. It suffices to restrict this analysis to the primes $p$ dividing $n\prod_{d\in A}d$, because otherwise $n_{p,a}=0 \ \forall a$ and $D_p=I$. By Lemma \ref{lem:reordersmith},
$\SF(F(C_G))=\prod_pD_p$ over $\Z$. Finally, \cite[Theorem A]{NoferiniWilliamsBrieskorn} gives
\eqref{eq:easySF}.
\end{proof}
\begin{example}
Let us apply Corollary \ref{thm:originale} by taking $f(t)=\Phi_4(t)\Phi_9(t)\Phi_{12}(t)$ and $n=6$. There are only two relevant primes. At $p=2$, we have $n_{2,1}=n_{2,3}=1$, $\mu_{2,1}=2$, and $\mu_{2,3}=4$, so $D_2=2I_6$. At $p=3$, we have $n_{3,1}=1$, $n_{3,2}=0$, and $\mu_{3,1}=\mu_{3,2}=3$, so $D_3=I_3\oplus3I_3$. Thus,
\[ \SF\left((\Phi_4\Phi_9\Phi_{12})(C_{t^6-1})\right) =2I_3\oplus6I_3.\qedhere \]
\end{example}

\section{Proofs of the main results}\label{sec:proof}

In this section we prove our main results. We begin by tackling Theorem \ref{thm:closed}, with the following strategy: We first analyze the case where  $G = \Phi_{p^b}$ is a single cyclotomic factor; then we prove how the Smith form of $F(C_G)$ changes with respect to that of $F(C_H)$ where $G=H \Phi_{p^c}$ and $p^c$ is larger than all the indices of the cyclotomic factors of $H$. Combining these two technical results yields \eqref{eq:closed}. Once Theorem \ref{thm:closed} is established, we apply it to prove Theorem \ref{thm:global}. For each prime $p$, we group the cyclotomic indices according to their coprime-to-$p$ parts and reduce each group to the prime-power case. Finally, Lemma \ref{lem:reordersmith} combines the resulting local Smith forms to give the Smith form over $\Z$.

We will need some basic number theory; for some background on valuations, local fields and field extensions, we refer to \cite[Chapter II, Sections 3--7]{Neukirch99}.

\subsection{Proof of Theorem \ref{thm:closed}}\label{sec:localproof}

Throughout this subsection, we fix a prime $p$, and all Smith forms and unimodular equivalences are over $\Z_p$. We will also often fix a nonnegative integer $c\geq 0$ and denote by $\zeta$ a primitive $p^c$-th root of unity in the algebraic closure of $\QQ_p$. For $c \geq 1$, we will also write $\O:=\Z_p[\zeta]$ and $\pi:=\zeta-1$, the uniformizer of $\O$; if instead $c=0$, then $\zeta:=1$, $\O:=\Z_p$, and $\pi:=p$.

Corollary \ref{cor:composition}, which is a special case of the Composition Theorem \cite[Theorem 3.13]{NoferiniWilliamsApplications}, will also be crucial.

 \begin{corollary}\label{cor:composition}
Let $f,g\in\Z_p[t]$, with $g$ monic, and define $h(t):=f(1+t)$ and $k(t):=g(1+t)$. Then $f(C_g)$ and $h(C_k)$ are unimodularly similar over $\Z_p$.
\end{corollary}

\subsubsection{Smith forms at a single cyclotomic factor}

For every integer $h\geq0$, define $H_h(t):=\Phi_{p^h}(1+t)$ and $e_h:=\varphi(p^h)$.
We now fix $c\geq0$. By \cite[Chapter II, Proposition 7.13]{Neukirch99}, $\O$ is a discrete valuation ring with uniformizer $\pi$, and $\QQ_p(\zeta)$ is an extension of $\QQ_p$ with degree and ramification index $e_c$ and inertia degree $1$. Hence $\O$ has residue field $\mathbb F_p$ and $\nu_\pi(p)=e_c$, where $\nu_\pi$ denotes the valuation normalized by $\nu_\pi(\pi)=1$. For $0\leq h<c$, we have
\begin{equation}\label{eq:cyclident}
 \nu_\pi\left(\Phi_{p^h}(\zeta)\right) =\nu_\pi\left(H_h(\pi)\right)=e_h.
\end{equation}
Indeed, $H_0(t)=t$. For $1\leq h<c$, we have $H_h(t)\equiv t^{e_h}\bmod\langle p\rangle$, so we may
write $H_h(t)=t^{e_h}+pR_h(t)$, with $R_h\in\Z_p[t]$. Observing that $\nu_\pi(\pi^{e_h})=e_h$
and  $\nu_\pi(p R_h(\pi))\geq e_c>e_h$, we get \eqref{eq:cyclident}; see also \cite[Lemma A.1]{Ohashi17}.

\begin{lemma}\label{lem:multiplication}
Let $\zeta$, $\pi$, and $e_c$ be as above, and let $f\in\Z_p[t]$ satisfy $f(\zeta)\ne0$. Consider the Euclidean division $\nu_\pi(f(\zeta))=:qe_c+r$, with $0\leq r<e_c$. Then,
\[\SF(f(C_{\Phi_{p^c}}))=p^q I_{e_c-r}\oplus p^{q+1}I_r.\]
\end{lemma}

\begin{proof}
If $c=0$, then the statement is obvious. We thus assume $c \geq 1$ with no loss of generality.

Observe that $\QQ_p(\zeta)$ is a field extension of degree $e_c \cdot 1 = e_c$ over $\QQ_p$, and that $\Phi_{p^c}(\zeta)=0$ with $\deg\Phi_{p^c}=e_c$. Hence, $\Phi_{p^c}$ is the minimal polynomial of $\zeta$. Furthermore, $\O=\Z_p[\zeta]$, implying that  $f(C_{\Phi_{p^c}}) \mapsto f(\zeta)$ is a ring isomorphism, i.e., $\Z_p[C_{\Phi_{p^c}}]\cong\O$. In particular, under this isomorphism we have $C_{\Phi_{p^c}}-I\mapsto \pi=\zeta-1$. On the other hand, defining $u:=f(\zeta)/(p^q\pi^r)$, it is clear that $\nu_\pi(u)=0$ and hence $u\in\O^\times$. Let $U\in\Z_p[t]$ satisfy $U(\zeta)=u$. The preimage of $u$ under the isomorphism is $V:=U(C_{\Phi_{p^c}})$, and therefore $V^{-1}$ is the preimage of $u^{-1}$, and in particular has elements in $\Z_p$. Thus, $V$ is unimodular and
\[ f(C_{\Phi_{p^c}}) =p^q V (C_{\Phi_{p^c}}-I)^r \sim p^q(C_{\Phi_{p^c}}-I)^r.\]
Since $\Phi_{p^c}(1)=p$, we have $\nu_p(\det(C_{\Phi_{p^c}}-I)^r)=r$, implying that the sum of the positive partial multiplicities of $(C_{\Phi_{p^c}}-I)^r$ is $r$. Moreover, $\Phi_{p^c}(t)\equiv(t-1)^{e_c}\bmod\langle p\rangle$
and $\gcd((t-1)^{e_c},(t-1)^r)=(t-1)^r$. Therefore, by \cite[Theorem 3.15]{NoferiniWilliamsApplications}, $(C_{\Phi_{p^c}}-I)^r$ has exactly $r$ nonunit invariant factors. We conclude that
\[ (C_{\Phi_{p^c}}-I)^r\sim I_{e_c-r}\oplus pI_r. \]
\end{proof}

Lemma \ref{lem:multiplication} immediately yields a special case of Theorem \ref{thm:closed}, when $G = \Phi_{p^c}$ is a single cyclotomic factor. For every $h>c\geq 0$, it holds $\zeta^{p^{h-1}}=1$ and hence $\Phi_{p^h}(\zeta)=\sum_{j=0}^{p-1} 1 = p$. Together with
\eqref{eq:cyclident}, this in turn implies
\[ \nu_\pi(F(\zeta)) =\sum_{\substack{h\in A\\h<c}}k_he_h  +e_c\sum_{\substack{h\in A\\h>c}}k_h
 =e_c(\alpha_c+q_c)+r_c.\]
Since $\delta_c=0$, Lemma \ref{lem:multiplication} gives \eqref{eq:closed} when $G=\Phi_{p^c}$.

\begin{example}
Let $p=3$ and $c=1$, so that $\Phi_3(t)=t^2+t+1$ and $e_1=2$. Then,
\[ C_{\Phi_3}-I=\begin{bmatrix}-2&-1\\1&-1\end{bmatrix}  \sim 1\oplus3,\]
matching the prediction of Lemma \ref{lem:multiplication} with $f(t)=t-1$, $q=0$, and $r=1$.
\end{example}

\begin{example}\label{ex:example1}
We now have the tools to prove \eqref{eq:example1}. Letting $f(t)=(t-1)^m$ and $g(t)=t^p-1$, it holds $\gcd(f,g)=t-1$, and hence $\SF(f(C_g))$ has one zero diagonal element. Moreover, $F(t)=(t-1)^{m-1}$ and $G(t)=\Phi_p(t)$. Applying Lemma \ref{lem:multiplication} to $F$, with $c=1$ and $\nu_\pi(F(\zeta))=m-1$, and noting that no prime other than $p$ is relevant to the global Smith form over $\Z$, we obtain \eqref{eq:example1}.
\end{example}

In our next developments, we will frequently use Lemma \ref{lem:multiplication}. Thus, it is convenient to introduce the following notation. For integers $c,v\geq 0$, we denote the corresponding list of partial multiplicities by
\begin{equation}\label{eq:uniformizer}
 \mathcal U_{e_c}(v):=(\underbrace{q,\ldots,q}_{e_c-r}, \underbrace{q+1,\ldots,q+1}_{r}), \qquad v=:qe_c+r,\quad 0\leq r<e_c.
\end{equation}
Note in particular that if $c=0$ then $\mathcal U_1(v)=(v)$.

\subsubsection{Adding a cyclotomic factor}

Our next aim is to study how the Smith form changes when a new cyclotomic factor (under the assumptions explained at the beginning of Section \ref{sec:proof}) is added to $G$. We first prove Theorem \ref{thm:collage}, and then obtain the general relation in Lemma \ref{thm:local}.

Note that, in the statement of Theorem \ref{thm:collage}, $H(C_Q)$ is never unimodular over $\Z_p$. Indeed, by \cite[Theorem 7.4]{NoferiniWilliamsBrieskorn} and using that the determinant is multiplicative, $\lvert\det H(C_Q)\rvert=\lvert\operatorname{Res}(H,Q)\rvert=p^{\deg H}$, and by assumption $\deg H>0$. Hence, $H(C_Q)^{-1}$ has entries in $\QQ_p$ and it is a priori nontrivial that the entries of $H(C_Q)^{-1}F(C_Q)$ are instead in $\Z_p$.

\begin{theorem}\label{thm:collage}
Let $c\geq1$, and let $A,B_-$ be disjoint subsets of $\{0,\ldots,c-1\}$. Given positive integers
$\{k_h\}_{h \in A}$, define
\[ F(t):=\prod_{h\in A}\Phi_{p^h}(t)^{k_h},\qquad  H(t):=\prod_{b\in B_-}\Phi_{p^b}(t),\qquad Q(t):=\Phi_{p^c}(t).\]
Suppose that $0<\deg H<\deg F$. Then $H(C_Q)^{-1}F(C_Q)$ has entries in $\Z_p$, and
\begin{equation}\label{eq:wedontprovethat}
 F(C_{HQ})\sim pF(C_H)\oplus H(C_Q)^{-1}F(C_Q).
\end{equation}
In particular, if the invariant factors of $F(C_H)$ are $p^{\eta_1},\ldots,p^{\eta_{\deg H}}$, then the partial multiplicities of $F(C_{HQ})$ over $\Z_p$ are obtained by joining the two lists
\[ (\eta_1+1,\ldots,\eta_{\deg H}+1),\qquad \mathcal U_{\varphi(p^c)}(\deg F-\deg H)\]
and arranging their entries in increasing order, where $\mathcal U_{\varphi(p^c)}$ is defined as in \eqref{eq:uniformizer}.
\end{theorem}

As a preliminary step, we need to establish the technical Lemmata \ref{lem:restocoeff}, \ref{lem:compablocchi} and \ref{lem:elimina}.
\begin{lemma}\label{lem:restocoeff}
Fix $c\geq1$, and let $\pi$ and $\O$ be defined as in Subsection \ref{sec:localproof}. Let $F$ and $H$ be products, with repetitions allowed, of polynomials of the form $\Phi_{p^h}(1+x)$ with $0\leq h<c$. Suppose that $\deg F\geq\deg H>0$, and write $R(x):=F(x)\bmod H(x)=\sum_{j=0}^{\deg H-1}a_jx^j$. Then
\begin{equation}\label{eq:restocoeff}
 \varphi(p^c)\nu_p(a_j)+j>\deg F\qquad(0\leq j<\deg H).
\end{equation}
In particular, $\nu_\pi(R(\pi))>\deg F$.
\end{lemma}
\begin{proof}
For notational convenience, we set within this proof $D:=\deg F$, $d:=\deg H$, and $e_h:=\varphi(p^h)$ for $0\leq h\leq c$. When $1\leq h<c$, the nonleading coefficients of $\Phi_{p^h}(1+x)$ are divisible by $p$, and $e_h<e_c=\nu_\pi(p)$. Thus, working in $\O[y]$,
\[ \pi^{-e_h} \Phi_{p^h}(1+\pi y) \equiv   y^{e_h} \bmod \langle \pi \rangle.\]
The same conclusion holds for $h=0$, because $\pi^{-1}\Phi_1(1+\pi y)=y$. It follows that $\pi^{-D}F(\pi y)$ and $\pi^{-d}H(\pi y)$ are monic polynomials in $\O[y]$, congruent to $y^D$ and $y^d$, respectively, modulo $\langle\pi\rangle$. Taking into account $D\geq d$, we conclude that every coefficient of
\begin{equation}\label{eq:allpimultiples}
 \pi^{-D}F(\pi y) - y^{D-d}\pi^{-d}H(\pi y)
\end{equation}
is a multiple of $\pi$. Consider the Euclidean polynomial division $a(y)=q(y)b(y)+r(y)$ where $b(y)$ is monic and the involved polynomials have coefficients in a GCD domain. By \cite[Lemma 7.3]{NoferiniWilliamsBrieskorn}, a common divisor of the coefficients of the dividend $a$ also divides every coefficient of the remainder $r$. Specializing this result to $\O[y]$, with monic divisor $\pi^{-d}H(\pi y)$ and dividend \eqref{eq:allpimultiples}, the remainder is $\pi^{-D}R(\pi y)$. Hence, each coefficient $\pi^{j-D}a_j$ is divisible by $\pi$, and thus $e_c\nu_p(a_j)+j-D>0$, proving \eqref{eq:restocoeff}. Finally, each addend $a_j\pi^j$ has valuation greater than $D$, and therefore the same holds for $R(\pi)$ which is their sum.
\end{proof}

Lemma \ref{lem:compablocchi} builds on, and enhances, \cite[Lemmata 5.1 and 5.4]{NoferiniWilliamsBrieskorn}.

\begin{lemma}\label{lem:compablocchi}
Fix $c\geq1$, and define $q(x):=\Phi_{p^c}(1+x)$. Let $b$ and $h$ be products, with repetitions allowed, of polynomials of the form $\Phi_{p^j}(1+x)$ with $0\leq j<c$. Suppose that $0<\deg h<\deg b$ and $\deg h\leq\varphi(p^c)$. Then $T:=h(C_q)^{-1}b(C_q)$ is nonsingular and belongs to $\Z_p[C_q]$. Moreover, there exist a unimodular matrix $V$ and a matrix $W\in\Z_p^{(\deg h)\times\varphi(p^c)}$ such that
\begin{equation}\label{eq:blocktriang}
 Vb(C_{hq})V^{-1}
 =\begin{bmatrix}b(C_h)&WT\\0&b(C_q)\end{bmatrix},
\end{equation}
and $\mathrm{rank}\left(\begin{bmatrix}W\\h(C_q)\end{bmatrix} \bmod \langle p \rangle \right) = \varphi(p^c)$.
\end{lemma}
\begin{proof}
We set $D:=\deg b$, $d:=\deg h$, and $e_c:=\varphi(p^c)$ within the proof. Recall that $\pi=\zeta-1$, where $\zeta$ is a primitive $p^c$-th root of unity. By \eqref{eq:cyclident},
$\nu_\pi(h(\pi))=d$ and $\nu_\pi(b(\pi))=D$. Hence $b(\pi)/h(\pi)\in\O$. Under the companion ring isomorphism \cite[Theorem 3.1]{NoferiniWilliamsBrieskorn} $\Z_p[C_q]\cong\O$, which sends $C_q$ to $\pi$, the preimage of this nonzero element is $T$. Thus $T\in\Z_p[C_q]$ is nonsingular.

To obtain \eqref{eq:blocktriang}, we define $\Lambda_k$ and $V$ by
\[
 \Lambda_k(x):=\begin{bmatrix}x^{k-1}\\\vdots\\1\end{bmatrix},
 \qquad
 V\Lambda_{d+e_c}(x)
 :=\begin{bmatrix}\Lambda_d(x)\\h(x)\Lambda_{e_c}(x)\end{bmatrix}.
\]
Since $h$ is monic and the bases $\Lambda_{d+e_c}(x)$ and (up to a row permutation) $V\Lambda_{d+e_c}(x)$ are both degree graded, $V$ is the product of a permutation matrix times a triangular matrix with unit diagonal; hence $V$ is unimodular.

For $0\leq i<d$, let $r_i:=(bx^i)\bmod h$, so that
\begin{equation}\label{eq:r} b(C_h) \Lambda_d(x) = \begin{bmatrix} r_{d-1}\\ \vdots\\ r_0\end{bmatrix}. \end{equation}
Since $x^i=[\Phi_1(1+x)]^i$, the polynomial $bx^i$ also has the form required by Lemma \ref{lem:restocoeff}, and its degree is $D+i$. Thus, by Lemma \ref{lem:restocoeff}, $\nu_\pi(r_i(\pi))>D+i$. We now define $w_i:=\pi^i-r_i(\pi)/b(\pi)$ for $i=0,\dots,d-1$. Then
\begin{equation}\label{eq:piadicvals}
 \nu_\pi(w_i-\pi^i) =\nu_\pi(r_i(\pi))-\nu_\pi(b(\pi)) >D+i-D=i.
\end{equation}
Hence $w_i-\pi^i\in\pi^{i+1}\O$, implying in turn $w_i\in\O$. Observe that $\O=\Z_p[\pi]$ is a free $\Z_p$-module, and it admits the monomial basis $\{\pi^{e_c-1},\ldots,\pi,1\}$.
Therefore, for $0\leq i<d$, $w_i$ has a unique $\Z_p$-linear expansion in this basis. Equivalently, there exists a  unique matrix $W\in\Z_p^{d\times e_c}$ such that
\begin{equation}\label{eq:W}
 W\Lambda_{e_c}(\pi)=\begin{bmatrix}w_{d-1}\\\vdots\\w_0\end{bmatrix}.
\end{equation}
Observing that $T\Lambda_{e_c}(\pi)=(b(\pi)/h(\pi))\Lambda_{e_c}(\pi)$, using also \eqref{eq:r} and \eqref{eq:W} we get
\[  WT\Lambda_{e_c}(\pi) =\frac{b(\pi)}{h(\pi)}W\Lambda_{e_c}(\pi) =\frac{b(\pi)\Lambda_d(\pi)-b(C_h)\Lambda_d(\pi)}{h(\pi)}.\]
Recall that $q(\pi)=\Phi_{p^c}(\zeta)=0$.  Hence, a polynomial in $\Z_p[x]$ vanishes at $\pi$ if and only if it is divisible by $q$. By \eqref{eq:r} and the definition of $r_i$, $b(x)\Lambda_d(x)-b(C_h)\Lambda_d(x)$ is divisible by $h$, and hence we conclude that
\[ \frac{b(x)\Lambda_d(x)-b(C_h)\Lambda_d(x)}{h(x)} \equiv WT\Lambda_{e_c}(x)\bmod\langle q\rangle.\]
Multiplying by $h(x)$, and using $b(x)\Lambda_{e_c}(x)\equiv b(C_q)\Lambda_{e_c}(x)
\bmod\langle q\rangle$, we obtain, modulo $\langle hq\rangle$,
\[ Vb(C_{hq})\Lambda_{d+e_c}(x) \equiv\begin{bmatrix} b(C_h)\Lambda_d(x)+h(x)WT\Lambda_{e_c}(x)\\
 h(x)b(C_q)\Lambda_{e_c}(x) \end{bmatrix} =\begin{bmatrix}b(C_h)&WT\\0&b(C_q)\end{bmatrix} V\Lambda_{d+e_c}(x).\]
This is a congruence of vectors having degree less than $d+e_c$, and hence it is in fact an equality, proving \eqref{eq:blocktriang}.

To prove the last part of the statement, we consider
$ S:=\begin{bmatrix}W\\\begin{bmatrix}0&I_{e_c-d}\end{bmatrix}h(C_q)\end{bmatrix}\in\Z_p^{e_c\times e_c}$.  $S$ is a square submatrix of $\begin{bmatrix}W\\h(C_q)\end{bmatrix}$, and we claim that it is nonsingular modulo $p$.
To prove the claim, note that, by \eqref{eq:W} and $h(C_q)\Lambda_{e_c}(\pi)=h(\pi)\Lambda_{e_c}(\pi)$,
\[ S\Lambda_{e_c}(\pi)= \left[\begin{array}{c}  \begin{matrix}w_{d-1}\\\vdots\\w_0\end{matrix}\\[0.5ex]  \begin{matrix}h(\pi)\pi^{e_c-d-1}\\\vdots\\h(\pi)\pi\\h(\pi)\end{matrix} \end{array}\right] \begin{array}{l} \left.\vphantom{\begin{matrix}w_{d-1}\\\vdots\\w_0\end{matrix}}\right\}  \ \in \O^{d}\\[0.5ex]  \left.\vphantom{\begin{matrix}h(\pi)\pi^{e_c-d-1}\\\vdots\\h(\pi)\pi\\h(\pi)\end{matrix}}\right\}
  \ \in \O^{e_c-d}
 \end{array}.\]
By \eqref{eq:piadicvals}, we have $\nu_\pi(w_i)=i$; moreover, $\nu_\pi(h(\pi))=d$. We conclude that elements of $S \Lambda_{e_c}(\pi)$ have pairwise distinct valuations in $\{0,\ldots,e_c-1\}$. Suppose for a contradiction that $S \bmod p$ is singular. Then, there exists a nonzero row vector $y^T \in \{0,\ldots,p-1\}^{1\times e_c}$ such that $y^TS\equiv0\bmod p$. Since $\Lambda_{e_c}(\pi)$ has entries in $\O$, we then have (recalling $\nu_\pi(p)=e_c$)
$y^TS\Lambda_{e_c}(\pi)\in p\O \Rightarrow \nu_\pi\left(y^TS\Lambda_{e_c}(\pi)\right)\geq e_c$.
On the other hand, all the nonzero entries of $y$ are units of $\Z_p$. Hence, we get the contradiction
\[ e_c \leq \nu_\pi(y^T S \Lambda_{e_c}(\pi) ) = \min_{j:y_j\neq 0} \nu_\pi((S \Lambda_{e_c}(\pi))_j) < e_c.\]
\end{proof}

\begin{lemma}\label{lem:elimina}
Let $0<d\leq e$, and let $A\in\Z_p^{d\times d}$, $B,T\in\Z_p^{e\times e}$, and $W\in\Z_p^{d\times e}$, with $T$ nonsingular. Suppose that (i) $\SF(B)=I_{e-d}\oplus pI_d$;
(ii) $\left[\begin{smallmatrix}W\\B\end{smallmatrix}\right]\in\Z_p^{(d+e)\times e}$ has full column rank modulo $p$; and (iii) there exists $k\in \mathbb{N}$ such that both $p^{-k}A$ and $p^kT^{-1}$ have entries in $\Z_p$. Then
\[ \begin{bmatrix}A&WT\\0&BT\end{bmatrix}\sim pA\oplus T.\]
\end{lemma}
\begin{proof}
By (ii) and \cite[Theorem 15 and Remark 17]{NoferiniInvertibleBases}, there exist $S_1\in\Z_p^{e\times d}$ and $S_2\in\Z_p^{e\times e}$ such that $S_1W+S_2B=I_e$. We define $U:=\begin{bmatrix}pI_d&-pWB^{-1}\\S_1&S_2\end{bmatrix}.$
By (i), $pB^{-1}$ has entries in $\Z_p$, so the same holds for $U$. Moreover, $ U\begin{bmatrix}I_d&W\\0&B\end{bmatrix}
 =\begin{bmatrix}pI_d&0\\S_1&I_e\end{bmatrix}.$ Thus, $\det U\det B=p^d$. Hence, using (i), $U$ is unimodular.

Multiplying times $A\oplus T$ from the right, we get
\[ U\begin{bmatrix}A&WT\\0&BT\end{bmatrix} =\begin{bmatrix}pA&0\\S_1A&T\end{bmatrix}.\]
Observe now that, by (iii),  $ T^{-1}S_1A=(p^kT^{-1})S_1(p^{-k}A) \in\Z_p^{e\times d}$. Collecting all these pieces, we have the unimodular equivalence
\[ U\begin{bmatrix}A&WT\\0&BT\end{bmatrix} \begin{bmatrix}I_d&0\\-T^{-1}S_1A&I_e\end{bmatrix} =pA\oplus T.\]
\end{proof}

\begin{proof}[Proof of Theorem \ref{thm:collage}]
We set $D:=\deg F$, $d:=\deg H$, and $e_c:=\varphi(p^c)$ within the proof. By Corollary \ref{cor:composition}
and \cite[Lemma 3.14]{NoferiniWilliamsApplications}, we can equivalently replace $F,H,Q$ by $b,h,q$, respectively, where $b(x):=F(1+x)$, $h(x):=H(1+x)$, and $q(x):=Q(1+x)$. In view of this,  \eqref{eq:wedontprovethat} is equivalent to
\begin{equation}\label{eq:weprovethat}
 b(C_{hq})\sim pb(C_h)\oplus h(C_q)^{-1}b(C_q).
\end{equation}
Since $H$ is squarefree, $d\leq\sum_{j=0}^{c-1}e_j=p^{c-1}\leq e_c$. By Lemma \ref{lem:compablocchi}, $T:=h(C_q)^{-1}b(C_q)\in\Z_p[C_q]$ and \eqref{eq:blocktriang} holds with $W$ as in \eqref{eq:W}, and $\begin{bmatrix}W\\h(C_q)\end{bmatrix}$ having full column rank modulo $p$. Moreover, by applying again Corollary \ref{cor:composition} and \cite[Lemma 3.14]{NoferiniWilliamsApplications}, we conclude also that $H(C_Q)^{-1}F(C_Q)\in\Z_p^{e_c\times e_c}$.

Our next move is to apply Lemma \ref{lem:elimina} to \eqref{eq:blocktriang}, with $A=b(C_h)$ and $B=h(C_q)$ so that $BT=b(C_q)$; this proves \eqref{eq:weprovethat} and hence \eqref{eq:wedontprovethat}. It remains to verify that the Lemma's assumptions hold. We have just observed that (ii) holds; for (i), recalling that $h(\pi)=H(\zeta)$, using \eqref{eq:cyclident} we get $\nu_\pi(h(\pi))=\sum_{j \in B_-} e_j = d$. Hence, by Lemma \ref{lem:multiplication} and Corollary \ref{cor:composition}, $\SF(h(C_q))=I_{e_c-d}\oplus pI_d$. For (iii), let us define $\delta:=\lceil(D-d)/e_c\rceil$ and write $b(x)\bmod h(x)=\sum_{j=0}^{d-1}a_jx^j$. Applying Lemma \ref{lem:restocoeff} to the pair $b,h$, we obtain $e_c\nu_p(a_j)+j>D$ by \eqref{eq:restocoeff}. Since $j<d$, we have $\nu_p(a_j)>(D-d)/e_c$, and hence $\nu_p(a_j)\geq\delta$. Thus every coefficient $a_j$ is divisible by $p^{\delta}$. Since $h(C_h)=0$, we have $b(C_h)=\sum_{j=0}^{d-1}a_jC_h^j$, so $p^{-\delta}b(C_h)\in\Z_p^{d\times d}$. Moreover, an argument similar to the one above yields $\nu_\pi(b(\pi))=D$, and hence $\nu_\pi(p^{\delta}h(\pi)/b(\pi))=\delta e_c-(D-d)\geq0$, so $p^{\delta}h(\pi)/b(\pi)\in\O$. The companion ring isomorphism $\Z_p[C_q]\cong\O$ sends $C_q$ to $\pi$ \cite[Theorem 3.1]{NoferiniWilliamsBrieskorn}. Under the inverse of this isomorphism, $p^{\delta}h(\pi)/b(\pi)$ corresponds to $p^{\delta}T^{-1}\in\Z_p[C_q]$. This verifies (iii) with $k=\delta$.

Finally, the partial multiplicities of $pb(C_h)$ are $\eta_1+1,\ldots,\eta_d+1$. The companion ring isomorphism identifies $T$ with $b(\pi)/h(\pi)$, whose $\pi$-adic valuation is $D-d$. Therefore, applying Lemma \ref{lem:multiplication} and Corollary \ref{cor:composition}, the partial multiplicities of $T$ are given by the list $\mathcal U_{e_c}(D-d)$ as defined in \eqref{eq:uniformizer}.
\end{proof}

\begin{example}
Let us consider $p=2$, $F(t)=\Phi_2(t)^2=(t+1)^2$, $H(t)=\Phi_1(t)=t-1$, and $Q(t)=\Phi_4(t)=t^2+1$.

Then $F(C_H)=4$ and $H(C_Q)^{-1}F(C_Q)=I_2-C_Q = \begin{bmatrix}
1 &1\\
-1&1
\end{bmatrix} \sim \begin{bmatrix}
1 & 0\\
0 & 2
\end{bmatrix}$ (the last step follows immediately by computing the gcd of the entries and the determinant). Thus, using \eqref{eq:wedontprovethat} and Lemma \ref{lem:reordersmith}, $\SF(F(C_{HQ})) = 1 \oplus 2 \oplus 8$.
\end{example}

Theorem \ref{thm:collage} describes the change in the Smith form when $0<\deg H<\deg F$ and all cyclotomic indices in $F$ are smaller than $p^c$. This does not yet cover all the possible pairs $F,G$ satisfying the assumptions of Theorem \ref{thm:closed}; we study the remaining cases in Lemma \ref{thm:local}.

\begin{lemma}\label{thm:local}
Let $F,G,A,B$ be defined as in Theorem \ref{thm:closed}. Let $c:=\max B$ and
\[ P(t):=\prod_{\substack{h\in A\\h<c}} \Phi_{p^h}(t)^{k_h},\qquad H:=G/\Phi_{p^c}.\]
Write $N:=\sum_{\substack{h\in A\\h>c}}k_h$, and the Euclidean division $\max(\deg P-\deg H,0)=:qe_c+r$, $0\leq r<e_c$. Let $p^{\eta_1},\ldots,p^{\eta_{\deg H}}$ be the invariant factors of $P(C_H)$. Then,
\begin{equation}\label{eq:recursion}
 F(C_G)\sim p^N\left(\left(\bigoplus_{i=1}^{\deg H} p^{\mu_i}\right)\oplus p^q I_{e_c-r}\oplus p^{q+1}I_r\right),
\end{equation}
where $\mu_i=0$ if $\eta_i=0$, and $\mu_i=\eta_i+1$ if $\eta_i>0$.
\end{lemma}

\begin{proof}
We set $D:=\deg P$ and $d:=\deg H$ within the proof. If $h>c$, then $G\mid t^{p^{h-1}}-1$, and hence $\Phi_{p^h}(t) = \Phi_p(t^{p^{h-1}}) \equiv p\bmod G(t)$. Thus, $\Phi_{p^h}(C_G)=pI_{\deg G}$ for $h>c$, and $F(C_G)=p^NP(C_G)$. In view of this, we may without loss of generality assume $N=0$ and $F=P$.

If $B=\{c\}$, then \eqref{eq:cyclident} yields $\nu_\pi(F(\zeta))=D$, where $\zeta$ is a primitive $p^c$-th root of unity. By Lemma \ref{lem:multiplication} and \eqref{eq:uniformizer}, the partial multiplicities of $F(C_G)$ are the elements of $\mathcal U_{e_c}(D)$.

We now assume $d>0$ and define $Q(t):=\Phi_{p^c}(t)$. We write $\overline F,\overline H\in\mathbb F_p[t]$ for the reductions of $F,H$ modulo $p$. There are two cases:
\begin{itemize}
\item If $D\leq d$, then $D\leq d\leq p^{c-1}\leq e_c$. By \eqref{eq:cyclident}, Lemma \ref{lem:multiplication}, and the Swap Theorem \cite[Theorem 3.11]{NoferiniWilliamsApplications}, we have
\[ I_{e_c-D}\oplus Q(C_F)\sim F(C_Q)\sim I_{e_c-D}\oplus pI_D.\]
By the uniqueness of the Smith form, $Q(C_F)\sim pI_D$, so $U:=p^{-1}Q(C_F)$ is unimodular over $\Z_p$. Applying the Swap Theorem twice, we have
\begin{align}
 F(C_{HQ})&\sim I_{d+e_c-D}\oplus H(C_F)Q(C_F)
             \sim I_{d+e_c-D}\oplus pH(C_F),\label{eq:small-swap1}\\
 F(C_H)&\sim I_{d-D}\oplus H(C_F).\label{eq:small-swap2}
\end{align}
Since $\overline F(t)=(t-1)^D$, $\overline H(t)=(t-1)^d$, and $D\leq d$, every partial multiplicity of $H(C_F)$ is positive by \cite[Theorem 3.15]{NoferiniWilliamsApplications}. Hence, \eqref{eq:recursion} follows from \eqref{eq:small-swap1} and \eqref{eq:small-swap2}.
\item If $D>d$, then $\overline H(t)=(t-1)^d$ divides $\overline F(t)=(t-1)^D$. Thus $F(C_H)\equiv0\bmod p$, so $\eta_i>0$ and $\mu_i=\eta_i+1$ for every $i$. The assumptions of Theorem \ref{thm:collage} hold, because $0<d<D$ and all the indices of the cyclotomic factors of $F=P$ are smaller than $p^c$. Hence, by Theorem \ref{thm:collage}, the partial multiplicities of $F(C_G)$ are obtained by joining the list $(\mu_1,\ldots,\mu_d)$ and $\mathcal U_{e_c}(D-d)$. Comparing with \eqref{eq:uniformizer} and noting that $D-d=qe_c+r$, this proves \eqref{eq:recursion}.
\end{itemize}
\end{proof}

\subsubsection{Proof of \texorpdfstring{\eqref{eq:closed}}{(\ref{eq:closed})}}

We can finally cash in on the previous technical accomplishments and prove Theorem \ref{thm:closed}.

\begin{proof}[Proof of Theorem \ref{thm:closed}] 
We construct $F$ and $G$ by adjoining their cyclotomic factors in increasing order of the indices in $A\cup B$. We denote the products already constructed by $P$ and $H$, respectively, starting with $P(t)=H(t)=1$. When $H=1$, the list of existing partial multiplicities is empty.

The recursion has two subcases:
\begin{itemize}
\item If the next index is $h\in A$, we replace $P \leftarrow P\cdot \Phi_{p^h}^{k_h}$. By construction, every index of the cyclotomic factors of $H$ is smaller than $p^h$, and hence $\Phi_{p^h}(C_H)=pI_{\deg H}$ by the same argument as in the proof of Lemma \ref{thm:local}. Thus every partial multiplicity of $P(C_H)$ increases by $k_h$.
\item If the next index is $b\in B$, we replace $H\leftarrow H\cdot \Phi_{p^b}$. We apply Lemma \ref{thm:local} with $F=P$, $G=H\Phi_{p^b}$, and $c=b$, where $H$ here denotes the product before the current step. These polynomials are coprime, and $H\Phi_{p^b}$ is squarefree. Every index of the cyclotomic factors of $P$ is, by construction, smaller than $p^b$, so the sum defining $N$ in the Lemma is empty, and $N=0$. Hence, positive partial multiplicities increase by $1$, and zero partial multiplicities remain zero. In addition, $e_b$ new partial multiplicities arise. To determine them, observe that the degrees of $P$ and $H$ before the current step are precisely the two sums in the left hand side of \eqref{eq:division}. Defining $q_b$ and $r_b$ accordingly, the new partial multiplicities are therefore  $q_b$ ( $e_b-r_b$ times) and $q_b+1$ ( $r_b$ times).
\end{itemize}
We now fix $b\in B$ and follow, as we add the next factors, the $e_b$ partial multiplicities that first arose when $b$ was added. Any partial multiplicity that is initially positive stays positive. Each later factor $\Phi_{p^h}^{k_h}$ of $F$ increases it by $k_h$, and each later factor $\Phi_{p^c}$ of $G$ increases it by $1$. Its total increase is therefore $\sum_{\substack{h\in A\\h>b}}k_h+\#\{c\in B:c>b\}=\alpha_b$. There are now again two cases:
\begin{itemize}
\item[$q_b \geq 1$] The $e_b-r_b$ partial multiplicities initially equal to $q_b$ thus become $\alpha_b+q_b$, while the $r_b$ initially equal to $q_b+1$ become $\alpha_b+q_b+1$. 
\item[$q_b=0$] The same argument above applies to the $r_b$ partial multiplicities initially equal to $1$; their final value is thus $\alpha_b+1$. Instead, a partial multiplicity initially equal to zero remains zero until we add the first cyclotomic factor of $F$ with index greater than $p^b$, say, $\ell > p^b$. There are $\delta_b$ cyclotomic factors of $G$ with indices between $p^b$ and $\ell$; they are added while the partial multiplicity is still zero and do not increase it. Every later factor of $G$ increases it by $1$, and its final value is therefore $\alpha_b-\delta_b$. (Note that, if there is no factor of $F$ with index greater than $p^b$, then $\alpha_b=\delta_b$, and the partial multiplicity remains zero.)
\end{itemize}
Finally, taking direct sums over $b\in B$ yields \eqref{eq:closed}.
\end{proof}

\subsection{Proof of Theorem \ref{thm:global}}

We now turn to the global Smith form over $\Z$. Our first ingredient is the following version of Hensel's Lemma \ref{lem:hensel}. To state it, recall our ``bar'' notation: Given the ring of integers $\O$ with residue field $k = \O/\mathfrak{m}$ and a polynomial $P(t) = \sum_i a_i t^i \in \O[t]$, we denote its reduction by $\overline{P}(t)=\sum_i [a_i] t^i \in k[t]$, where $ [a_i] = a_i \bmod \mathfrak{m}$.

\begin{lemma}[Hensel's Lemma, {\cite[Chapter II, (4.6)]{Neukirch99}}]
\label{lem:hensel}
Let $K$ be a finite extension of $\QQ_p$, with ring of integers $\O$, maximal ideal $\mathfrak m$, and residue field $k:=\O/\mathfrak m$. Let the monic polynomial $P\in\O[t]$ satisfy $\overline P= \prod_{i=1}^s\overline P_i$, where the $\overline P_i\in k[t]$ are monic and pairwise coprime polynomials. Then, there are unique monic polynomials $P_i\in\O[t], i=1,\dots,s,$ such that
\begin{enumerate}
\renewcommand{\labelenumi}{(\roman{enumi})}
\item $P_i\equiv\overline P_i\bmod\mathfrak m$ for $i=1,\ldots,s$;
\item $\deg P_i=\deg\overline P_i$ for $i=1,\ldots,s$; and
\item $P=P_1\cdots P_s$.
\end{enumerate}
In particular, every simple root of $\overline P$ in $k$ is congruent modulo $\mathfrak m$ to a unique root of $P$ in $\O$.
\end{lemma}

We next need a matrix theoretical consequence of B\'{e}zout's identity. Note in particular that Lemma \ref{lem:bezout} applies to the case $\O=\Z_p, \mathfrak{m} = \langle p \rangle$.

\begin{lemma}\label{lem:bezout}
In the notation of Lemma \ref{lem:hensel}, suppose that $P,H \in \O[t]$, with $H$ monic, are such that $ \overline{P}, \overline{H} $ are coprime. Then, $\operatorname{Res}(P,H)$ is a unit of $\O$, or equivalently, the matrix $P(C_H)$ is unimodular over $\O$. 
\end{lemma}

\begin{proof}
By \cite[Equation (1.1)]{NoferiniWilliamsBrieskorn}, we have
\[ \det P(C_H)=\operatorname{Res}(P,H) \equiv\operatorname{Res}(\overline P,\overline H)  \bmod\mathfrak m.\]
Since $\overline{P}$ and $\overline{H}$ are coprime by assumption, $\operatorname{Res}(P,H)=\det P(C_H)$ is a unit modulo $\mathfrak{m}$, and a fortiori a unit of the local ring $\O$. 
\end{proof}

As a next step, Theorem \ref{lem:bello} is similar in spirit to \cite[Theorem 6.3]{NoferiniWilliamsBrieskorn}, but it makes different assumptions and its statement yields a similarity over the base ring (as opposed to a weaker unimodular equivalence).

\begin{theorem}\label{lem:bello}
Let $R$ be a commutative ring (with unity), and let $H_1,H_2\in R[t]$ be monic polynomials. If $\operatorname{Res}(H_1,H_2)$ is a unit of $R$, then, for every $P\in R[t]$, there exists a unimodular $U$, independent of $P$, such that $ UP(C_{H_1H_2})=  (P(C_{H_1}) \oplus P(C_{H_2}))U$. In particular, $P(C_{H_1H_2}) \sim P(C_{H_1}) \oplus P(C_{H_2})$.
\end{theorem}
\begin{proof}
We may assume with no loss of generality that $H_1,H_2$ both have positive degree. Indeed, if $H_i=1$ then $C_{H_i}$ is the $0 \times 0$ empty matrix, and the statement holds trivially with $U=I$.

Denote for simplicity $X:=C_{H_1H_2}$ and  $Y:=C_{H_1} \oplus C_{H_2}$. It suffices to prove that $UX=YU$ for a unimodular $U$; indeed, the statement then follows by induction if $P$ is a monomial and by linearity for a more general polynomial. Let $n_i:=\deg H_i$, for $i=1,2$, and
$n:=n_1+n_2$, and denote \[\Lambda_k(t) := \begin{bmatrix}t^{k-1}\\
\vdots\\
t\\
1\end{bmatrix} \in R[t]^k.\] Let
$U\in R^{n\times n}$ be the Sylvester matrix of $H_1$ and $H_2$, which is defined by
$ U \Lambda_n(t) = \begin{bmatrix}
 H_2 \Lambda_{n_1}(t)\\
 H_1 \Lambda_{n_2}(t)
 \end{bmatrix}.$
We then have $\det U=\operatorname{Res}(H_1,H_2)$ \cite[Chapter 3]{IdealsVarietiesAlgorithms}, and therefore $U$ is unimodular. Moreover, by the properties of companion matrices recalled in the introduction, it holds
\[ UX\Lambda_n(t)\equiv tU\Lambda_n(t) \equiv\begin{bmatrix}
 tH_2\Lambda_{n_1}(t)\\
 tH_1\Lambda_{n_2}(t)
 \end{bmatrix} \equiv Y\begin{bmatrix}
 H_2\Lambda_{n_1}(t)\\
 H_1\Lambda_{n_2}(t)
 \end{bmatrix} \equiv YU\Lambda_n(t)\bmod\langle H_1H_2\rangle.\]
It follows that every entry of $(UX-YU)\Lambda_n(t)$, which has degree at most $n-1$, is divisible by the monic polynomial $H_1H_2$, which has degree $n$. Thus, $UX-YU=0$.
\end{proof}
Our last ingredient is the observation that the Smith form is invariant under a ring extension.
\begin{remark}\label{rem:invariance}
Let $\O$ be the ring of integers of a finite extension of $\QQ_p$. Then, every matrix $A\in\Z_p^{n\times n}$ has the same Smith form over $\Z_p$ and over $\O\supseteq\Z_p$. Indeed, let $S=UAV$ be the Smith form of $A$ over $\Z_p$; then, $U,V$ are also unimodular over $\O$ and thus $S$ is also the Smith form over $\O$, with the nonzero invariant factors normalized as powers of $p$ over both rings. (For the proof of Theorem \ref{thm:global}, we only need to consider unramified field extensions of $\QQ_p$.)\end{remark}

We are finally ready to cook the main dish.
\begin{proof}[Proof of Theorem \ref{thm:global}]
By \cite[Theorem A]{NoferiniWilliamsBrieskorn}, we have $f(C_g)\sim F(C_G)\oplus0_{\ell\times\ell}$. If $m=0$, $f(C_g)=0$ and there is nothing to prove. Otherwise, $F$ and $G$ are coprime, and $G$ is squarefree by hypothesis. Let us fix a prime $p$. The cyclotomic identities in \cite[p. 160]{Nagell01} give, for every $p\nmid a$ and $h\geq0$,
\begin{equation}\label{eq:cyclotomic-reduction}
 \overline{\Phi_{ap^h}}=\overline{\Phi_a}^{\,\varphi(p^h)}.
\end{equation}
The roots of $\overline{\Phi_a}$ have multiplicative order exactly $a$. Indeed, $t^a-1$ is squarefree over $\mathbb F_p$, because its derivative is $at^{a-1}$ and $p\nmid a$. In the
factorization $t^a-1=\prod_{d\mid a}\overline{\Phi_d}(t)$, a root of $\overline{\Phi_a}$ therefore cannot also be a root of $t^d-1$ unless $d=a$.

For every positive integer $a$ coprime to $p$, we now define $G_a(t):=\prod_{b\in B_{p,a}}\Phi_{ap^b}(t)\in\Z_p[t]$. By \eqref{eq:cyclotomic-reduction}, every root of $\overline{G_a}\in\mathbb F_p[t]$ has multiplicative order $a$. Hence, $\overline{G_a}$ and $\overline{G_{a'}}$ are coprime if $a\ne a'$, and thus $\operatorname{Res}(G_a,G_{a'})$
is a unit of $\Z_p$ by Lemma \ref{lem:bezout}. By repeated applications of Theorem \ref{lem:bello} for the ring $R=\Z_p$, we obtain
\[ F(C_G)\sim\bigoplus_{\substack{a\geq1\\p\nmid a,\ G_a\ne1}}F(C_{G_a}).\]
If a cyclotomic factor $P$ of $F$ has coprime-to-$p$ part $a'\ne a$, every root of $\overline P$ has multiplicative order $a'$, whereas every root of $\overline{G_a}$ has multiplicative order $a$. Thus $\overline P$ and $\overline{G_a}$ are coprime. By Lemma \ref{lem:bezout}, the matrix $P(C_{G_a})$ is therefore unimodular. Hence, factoring $F(C_{G_a})$,
\[ F(C_{G_a})\sim F_a(C_{G_a}),\qquad F_a(t):=\prod_{h\in A_{p,a}}\Phi_{ap^h}(t)^{k(ap^h)}.\]
It therefore suffices to determine the Smith form of $F_a(C_{G_a})$ for each $a$ with $G_a\ne1$.
For a fixed $a$, we choose a finite unramified extension of $\QQ_p$ whose residue field contains the $a$-th roots of unity \cite[Chapter II, Proposition 7.12]{Neukirch99}, and let $\O$ be its ring of integers. Since $p\nmid a$, the polynomial $t^a-1\in(\O/p\O)[t]$ has only simple roots. By Hensel's Lemma \ref{lem:hensel}, all $a$-th roots of unity lie in $\O$, and distinct roots have distinct reductions modulo $p$. By Remark \ref{rem:invariance}, we may equivalently study the Smith form of $F_a(C_{G_a})$ over $\O$, keeping the normalization in powers of $p$.

For each primitive $a$-th root of unity $\zeta$, let $P_{h,\zeta}(t):=\zeta^{e_h}\Phi_{p^h}(t/\zeta) \in \O[t]$, where $e_h:=\varphi(p^h)$. Denoting by $\overline{\zeta}$ the image of $\zeta$ in $\O/p\O$,
\[\Phi_{ap^h}(t)=\prod_{\zeta:\Phi_a(\zeta)=0}P_{h,\zeta}(t),\qquad \overline{P_{h,\zeta}}(t)=(t-\overline\zeta)^{e_h}.\]
We define $G_{a,\zeta}(t):=\prod_{b\in B_{p,a}}P_{b,\zeta}(t)$. For $\zeta\ne\zeta'$, the polynomials $\overline{G_{a,\zeta}}, \overline{G_{a,\zeta'}} \in (\O/p\O)[t]$ are powers of distinct linear factors, and hence coprime. Thus $\operatorname{Res}(G_{a,\zeta},G_{a,\zeta'})$
is a unit of $\O$ by Lemma \ref{lem:bezout}. Applying Theorem \ref{lem:bello} for the ring $R=\O$, we thus obtain
\[ F_a(C_{G_a})\sim \bigoplus_{\zeta:\Phi_a(\zeta)=0}F_a(C_{G_{a,\zeta}}).\]
When $\zeta'\ne\zeta$, we have by Lemma \ref{lem:bezout} that $P_{h,\zeta'}(C_{G_{a,\zeta}})$ is unimodular. Therefore, for the purpose of computing the Smith form it suffices to restrict to only the factors indexed by $\zeta$ in each product defining $F_a(C_{G_{a,\zeta}})$.

Let \[H(t):=\prod_{h\in A_{p,a}}\Phi_{p^h}(t)^{k(ap^h)},\qquad K(t):=\prod_{b\in B_{p,a}}\Phi_{p^b}(t).\]
By definition, $G_{a,\zeta}(t)=\zeta^{\deg K}K(t/\zeta)$. We therefore have
\begin{equation}\label{eq:laststep}
 D_\zeta^{-1}C_{G_{a,\zeta}}D_\zeta=\zeta C_K, \qquad D_\zeta:=\bigoplus_{j=1}^{\deg K}\zeta^{\deg K-j};
\end{equation}
observe in particular that $D_\zeta$ is unimodular, because $\zeta $ is a unit of $\O$, and therefore \eqref{eq:laststep} expresses a similarity of matrices over $\O$.  Moreover, $P_{h,\zeta}(\zeta C_K)=\zeta^{e_h}\Phi_{p^h}(C_K)$, and hence
\begin{equation}\label{eq:diagonalscaling}
 D_\zeta^{-1} \left(\prod_{h\in A_{p,a}}P_{h,\zeta}(C_{G_{a,\zeta}})^{k(ap^h)}\right)
 D_\zeta=\zeta^{\deg H}H(C_K),\end{equation}
implying that the product on the left hand side has the same Smith form as $H(C_K)$. Our last step is now to invoke Theorem \ref{thm:closed}, yielding $H(C_K)\sim D_{p,a}$ over $\Z_p$. Hence, by Remark \ref{rem:invariance}, the left hand side of \eqref{eq:diagonalscaling} is equivalent to $D_{p,a}$ over $\O$ (recall that $\zeta\in\O^\times$). There are precisely $\varphi(a)$ primitive $a$-th roots, and thus $F_a(C_{G_a})\sim\bigoplus_{j=1}^{\varphi(a)}D_{p,a}$ over $\O$. Both matrices have entries in
$\Z_p$; applying Remark \ref{rem:invariance} once more, they have the same Smith form over $\Z_p$. We combine these decompositions over all relevant $a$ and restore the $\ell$ zero entries to obtain \eqref{eq:globallocal}; finally, \eqref{eq:global} follows by Lemma \ref{lem:reordersmith}.
\end{proof}

\section{Applications}\label{sec:applications}
In this section, we discuss various applications.

\subsection{Groups, knots, branched covers, and graphs}

\subsubsection{Cyclically presented groups of type
\texorpdfstring{$\ZZ$}{Z}}

Let us start with a family of cyclically presented groups, called groups of Type $\ZZ$; these are of interest here because, under a particular hypothesis, their representer polynomial can be written as a product of cyclotomic polynomials. Moreover, the results obtained here will be useful later on, when we study knots in lens spaces and graphs. In \cite{McDermott22}, McDermott showed that, under suitable geometric assumptions, the $2$-dimensional complexes associated with the cyclic presentations of these groups can be realized as spines of closed $3$-dimensional manifolds, including the generalized Sieradski manifolds \cite{CHK} and certain Brieskorn manifolds \cite{Milnor}. Generally, the presentations of groups of type $ \ZZ$ are defined by a certain congruence modulo $n$ \cite[Lemma 4]{McDermott22}; below, we study a subfamily for which the congruence is replaced by an equality, and that can be parametrized as follows. Let $r>|s|>0$, let $q\geq 1$, and pick an integer $A$ satisfying 
\begin{equation}\label{eq:McDermott}
    2A+(r-s)q=0.
\end{equation}
Defining $a:=(r+s)/2$ and $b:=(r-s)/2$, and noting that \eqref{eq:McDermott} implies that $aq,bq \in \Z$, the presentations have representer polynomial $f(t)=(t^{aq}+1)(t^{bq}-1)/(t^q-1)$. When $r+s$ is even and $q=1$ Corollary \ref{cor:typeZeven} provides the abelianization of these groups, and hence (under the assumptions mentioned above) the first homology of the associated manifolds.
\begin{corollary}\label{cor:typeZeven}
Let $a,b,n\geq1$, $f(t):=(t^a+1)(t^b-1)/(t-1)$ and $g(t):=t^n-1$. Define $d:=\gcd(a,n)$, $e:=\gcd(b,n)$,
$h:=\gcd(a,b,n)$, and $L:=nh/(de)$. Then
\begin{equation}\label{eq:typeZ}
\operatorname{coker}f(C_g)\cong
\begin{cases}
\Z^{e-1}\oplus(\Z/2\Z)^{d-h}\oplus\Z/(2b/e)\Z,
 &\text{if }\nu_2(n)\leq\nu_2(a),\\
\Z^{d+e-1}\oplus\Z/(b/e)\Z,
 &\text{if }\nu_2(b)\leq\nu_2(a)<\nu_2(n),\\
\Z^{d+e-h-1}\oplus(\Z/L\Z)^{h-1}\oplus\Z/(2bL/e)\Z,
 &\text{if }\nu_2(a)<\min\{\nu_2(b),\nu_2(n)\}.
\end{cases}
\end{equation}
\end{corollary}
\begin{proof}
Observe that \[t^a+1=\prod_{j\mid2a,\,j\nmid a}\Phi_j(t) \qquad \mathrm{and} \qquad \frac{t^b-1}{t-1}=\prod_{j\mid b,\,j>1}\Phi_j(t).\] Let us distinguish the three cases in \eqref{eq:typeZ}.
\begin{enumerate}
\item If $ \nu_2(n)\leq \nu_2(a)$, $t^a+1$ is coprime to $g$, and hence $\gcd(f,g)=(t^e-1)/(t-1)$, giving $e-1$ zero diagonal elements in the Smith form of $f(C_g)$. In particular, every common cyclotomic factor has multiplicity $1$ in $f$, and thus we may apply Corollary \ref{thm:originale}. In particular, associated with the prime $2$ there are $d-h$ partial multiplicities equal to $1$ at $2$ and $1$ partial multiplicity equal to $1+\nu_2(b/e)$. At an odd prime $p$, the only possibly positive partial multiplicity is $\nu_p(b/e)$. We obtain the statement by gluing this local information.
\item When $\nu_2(b)\leq\nu_2(a)<\nu_2(n)$, $f$ is squarefree and its gcd with $g$ is $(t^d+1)(t^e-1)/(t-1)$, whose degree is $d+e-1$. Again by Corollary \ref{thm:originale}, the only possibly positive partial multiplicity is $\nu_p(b/e)$, for every prime $p$; this yields the claimed formula.
\item Suppose now that $ \nu_2(a)<\nu_2(b)$ and $\nu_2(a) < \nu_2(n)$. In this case, $ \gcd(f,g) = (t^d+1)(t^e-1)/((t^h+1)(t-1))    $ which determines that the free part of the cokernel of $f(C_g)$ is isomorphic to $\Z^{d+e-h-1}$. To analyze the torsion part, we apply Theorem \ref{thm:global} to $F=f/\gcd(f,g)$, $G=g/\gcd(f,g)$. Fix first an odd prime $p$ and write $\alpha:=\nu_p(a)$, $\beta:=\nu_p(b)$, and $\gamma:=\nu_p(n)$. The factors with coprime-to-$p$ part $1$ contribute a partial multiplicity $\nu_p(b/e)=\max(\beta-\gamma,0)$. Other positive partial multiplicities can only arise when $\gamma>\max(\alpha,\beta)$, and the corresponding coprime-to-$p$ parts are precisely the indices of the cyclotomic factors of $t^{h/p^{\min(\alpha,\beta)}}+1$. In turn, for each of these coprime-to-$p$ parts, the index sets associated with $F$ and $G$ are, respectively, $\{0,\ldots,\min(\alpha,\beta)\}$  (all with multiplicity $1$) and $\{\max(\alpha,\beta)+1,\ldots,\gamma\}$. By Theorem \ref{thm:closed}, each such index $j$ is associated with $p^{\min(\alpha,\beta)}$ positive partial multiplicities, all equal to $\nu_p(L)$. By Theorem \ref{thm:global}, each of these groups of partial multiplicities appears  $\varphi(j)$ times. Since $\sum_j \varphi(j)=h/p^{\min(\alpha,\beta)}$, overall we obtain $h$ partial multiplicities equal to $\nu_p(L)$. Note also that $\nu_p(L)$ and $\nu_p(b/e)$ cannot both be positive.

The bookkeeping differs at $p=2$. Still denoting $\alpha:=\nu_2(a), \beta:=\nu_2(b), \gamma:=\nu_2(n)$, we have that $\alpha$ is strictly smaller than both $\beta$ and $\gamma$. There are two subcases:
\begin{itemize}
\item Suppose first $\gamma>\beta$. For the coprime-to-$2$ parts not equal to $1$, which are the integers $j > 1$ that divide $2^{-\alpha} h$, the set of indices associated with $F$ is $\{\alpha+1\}$ and that associated with $G$ is $\{\beta+1,\ldots,\gamma\}$. By Theorem \ref{thm:global}, this gives $h-2^\alpha$ partial multiplicities equal to $\gamma-\beta$. For the odd part $1$, the set of indices associated with $G$ becomes $\{0,\beta+1,\ldots,\gamma\}$, and this results in $1$ partial multiplicity equal to $\gamma-\beta+1$ and $2^\alpha-1$ partial multiplicities equal to
$\gamma-\beta$.
\item If $\gamma\leq\beta$, only the coprime-to-$2$ part $1$ can generate positive partial multiplicities. The corresponding sets of indices are $\{\alpha+1,\gamma+1,\ldots,\beta\}$ for $F$ and
$\{0\}$ for $G$, and they result in $1$ partial multiplicity equal to $1+\beta-\gamma$.
\end{itemize}
In summary, at every prime, the positive partial multiplicities are the positive elements of the list consisting of $h-1$ copies of $\nu_p(L)$ and one of $\nu_p(2bL/e)$. Observing that $L\mid2bL/e$ and using Lemma \ref{lem:reordersmith}, we obtain the third case of the statement.
\end{enumerate}
\end{proof}

\begin{example}
For example, when $(q,a,b,n)=(1,3,6m+2,3)$, where $m \in \mathbb{N}$, Corollary \ref{cor:typeZeven} implies that the abelianization of the  corresponding group of type $\ZZ$ is isomorphic to $(\Z/ 2\Z)^2\oplus\Z/(12m+4)\Z$. 

These groups satisfy the assumptions in \cite{McDermott22}, and comparing with \cite[Corollary D]{NoferiniWilliamsBrieskorn}, we conclude that the associated manifolds are not cyclic branched covers of torus knots. Indeed, \cite[Corollary D]{NoferiniWilliamsBrieskorn} implies that the nontrivial invariant factors of a Brieskorn manifold $M(r,s,n)$ with $\gcd(r,s)=1$ and finite first homology are all equal. They are also not odd-degree cyclic covers of $S^3$ branched over a knot, because Plans' Theorem (see
\cite[p.~325]{Prz88} and the references therein) says that the first homology of such a manifold is a direct sum of two isomorphic groups; this is impossible for the family at hand, because the positive partial multiplicities of the prime $2$ are $\nu_2(12m+4),1,1$.
\end{example}

If $r+s$ is odd, the equality $(r-s)q+2A=0$ implies that $q$ is even. In Corollary \ref{cor:typeZodd} we study the case $q=2$. Letting $c:=r+s$ and $d:=r-s$, we obtain the representer polynomial $f(t)=(t^c+1)(t^d-1)/(t^2-1)$, where $c,d$ are positive odd integers.

\begin{corollary}\label{cor:typeZodd}
Let $c,d,n\geq1$, with $c,d$ odd, $f(t):=(t^c+1)(t^d-1)/(t^2-1)$, and $g(t):=t^n-1$. Let $\delta:=\gcd(c,n)$, $\varepsilon:=\gcd(d,n)$, and $h:=\gcd(c,d,n)$. Then,
\begin{equation}\label{eq:typeZodd}
\operatorname{coker}f(C_g)\cong
\begin{cases}
\Z^{\varepsilon-1}\oplus(\Z/2\Z)^{\delta-h}\oplus\Z/(d/\varepsilon)\Z,
 &\text{if} \ n \ \text{is odd},\\
\Z^{\delta+\varepsilon-2}\oplus\Z/(c/\delta)\Z\oplus\Z/(d/\varepsilon)\Z,
 &\text{if} \ n \ \text{is even}.
\end{cases}
\end{equation}
\end{corollary}
\begin{proof}
Since $c,d$ are odd, we have that 
\[ f(t)=\prod_{j\mid c,\,j>1}\Phi_{2j}(t) \prod_{j\mid d,\,j>1}\Phi_j(t).\]
In particular, $f$ is squarefree, and therefore we may apply Corollary \ref{thm:originale}. Let us first compute the free part of the module, which is associated with the gcd of $f$ and $g$. If $n$ is odd, then $\gcd(f,g)=(t^\varepsilon-1)/(t-1)$; and if $n$ is even, then $\gcd(f,g)=(t^\delta+1)(t^\varepsilon-1)/(t^2-1)$.

Let us now turn to the torsion part. At an odd prime $p$, the factors with coprime-to-$p$ part $1$ are associated with possibly positive partial multiplicity $\nu_p(d/\varepsilon)$. If $n$ is even, the factors with coprime-to-$p$ part $2$ also contribute $\nu_p(c/\delta)$. For any other coprime-to-$p$ part
$j$, the relevant set of indices of cyclotomic factors in $f$ is of the form $\{j p^i\}_{i=0}^k$, for some $k$, and the set of indices of factors of $g$ has the form $\{j p^i\}_{i=0}^{\nu_p(n)}$. Hence, one of these two sets must be contained in the other. Equivalently, either all such factors of $g$ divide $f$, or all such factors of $f$ divide $g$ (or both); thus, by Corollary \ref{thm:originale}, there is no associated positive partial multiplicity. For the prime $p=2$, having fixed a coprime-to-$2$ part $j$ that divides $n$, the only possible associated indices among the factors of $f$ are  $j$ and $2j$. If $n$ is even, both indices divide $n$, so there are no associated positive partial multiplicities. If $n$ is odd, a positive partial multiplicity arises precisely when $\Phi_{2j}$ divides $f$, while $\Phi_j$ divides $g$ but not $f$. In other words, it must happen that $j\mid\delta$ and $j\nmid h$. For each such $j$, Corollary \ref{thm:originale} gives $\varphi(j)$ partial multiplicities equal to $1$. Their total number is $\sum_{j\mid\delta,\,j\nmid h}\varphi(j)=\delta-h$.

Combining this local information via Lemma \ref{lem:reordersmith} we obtain \eqref{eq:typeZodd}.
\end{proof}

Finally, in Corollary \ref{cor:Zeveryq} we allow $q$ to take any value.

\begin{corollary}\label{cor:Zeveryq}
Let $n,q\geq1$ and $r,s,A\in\Z$ satisfy $r>|s|>0$ and $2A+(r-s)q=0$. Then, for the abelianization $G^{ab}$ of the group of type $\ZZ$ with these parameters, the following holds.
\begin{enumerate}
\item If $r+s$ is even, let $\kappa:=\gcd(n,q)$. Then, $G^{ab}$ is isomorphic to the direct sum of $\kappa$ copies of the group in \eqref{eq:typeZ}, with $a=(r+s)/2$, $b=(r-s)/2$, and $n$ replaced by $n/\kappa$.
\item If $r+s$ is odd, let $\kappa:=\gcd(n,q/2)$. Then, $G^{ab}$ is isomorphic to the direct sum of $\kappa$ copies of the group in \eqref{eq:typeZodd}, with $c=r+s$, $d=r-s$, and $n$ replaced by $n/\kappa$.
\end{enumerate}
\end{corollary}

\begin{proof}
Observe that the permutation matrix $C_{t^n-1}^q$ has $\gcd(n,q)$ cycles, each of length $n/\gcd(n,q)$; a similar remark holds for the second case, except that we replace $q$ by $q/2$. Then, the statement follows by combining Corollary \ref{cor:typeZeven}, Corollary \ref{cor:typeZodd}, and the Composition Theorem \cite[Theorem 3.13]{NoferiniWilliamsApplications}. 
\end{proof}

\subsubsection{Cyclic branched covers of
\texorpdfstring{$S^3$}{the 3-sphere}}

Let $\Lambda:=\Z[t,t^{-1}]$, and denote by $\Sigma_n(K)$ the $n$-fold cyclic cover of $S^3$, branched over the knot $K$. If $K$ has cyclic Alexander module $\Lambda/\langle f\rangle$, then the first homology of its $n$-fold cyclic branched cover is $\Lambda/\langle f,t^n-1\rangle$ \cite[Section 2]{SilverWilliamsTorsion}; see also \cite{NoferiniWilliamsBrieskorn}.

A known sufficient condition \cite[Corollary 1.2]{KY10} for the Alexander module to be cyclic is that $K$ has tunnel number \cite[p.~113]{MSK96} at most $1$. A cyclic Alexander module is isomorphic to $\Lambda/\langle\Delta_K\rangle$, where $\Delta_K$ denotes the Alexander polynomial
\cite[Lemma 3.1]{KY10}. We may assume (up to multiplication by a signed power of $t$, which is invertible $\bmod t^n-1$) that $\Delta_K$ is a polynomial in $t$; the trivial case $\Delta_K(t)=1$ corresponds to the zero module. Hence, for every knot $K$ of tunnel number at most $1$ and for every $n\geq1$, we have for the first homology of $\Sigma_n(K)$
\begin{equation}\label{eq:cchomology}
 H_1(\Sigma_n(K))\cong\Z[t]/\langle\Delta_K,t^n-1\rangle \cong\operatorname{coker}\Delta_K(C_{t^n-1}).
\end{equation}
When $\Delta_K$ is a product of cyclotomic polynomials, Theorem \ref{thm:global} therefore yields the full first homology, with no further restriction on the Alexander polynomial.

For example, $(1,1)$-knots in $S^3$ have tunnel number at most $1$ \cite[Section 4]{GMM05}. A $(1,1)$-knot is a knot in a manifold obtained by gluing two solid tori, meeting each solid torus in a boundary-parallel arc
\cite[Section 1]{Cattabriga06}. A palindromic polynomial $f\in\Z[t]$ is, up to multiplication by signed powers of $t$, the Alexander polynomial of a knot in $S^3$ if and only if $f(1)=\pm1$ \cite[Theorem 3.1]{Cattabriga06}. Using also \cite[Theorem 1 and concluding remark]{Fujii96}, such an Alexander polynomial can always be realized by a $(1,1)$-knot in $S^3$. Hence, this class is a rich source of applications for our results. One example was studied in \cite[Corollary D]{NoferiniWilliamsBrieskorn}: For coprime $r,s$, the torus knot $T(r,s)$ has the Alexander polynomial $\Delta_{T(r,s)}(t)=(t^{rs}-1)(t-1)/((t^r-1)(t^s-1))$  which, in contrast to the examples below, is squarefree.

\begin{example}\label{ex:Kr}
By the above comments, for every $r\geq1$, there is a $(1,1)$-knot $K_r\subset S^3$ with Alexander polynomial $\Phi_6^r$. For the $6$-fold cyclic cover of $S^3$ associated with such a knot, Theorem \ref{thm:global} and \eqref{eq:cchomology} yield
\[ H_1(\Sigma_6(K_r)) \cong\Z^2\oplus\Z/2^{r-1}\Z\oplus\Z/6^{r-1}\Z.\]
Indeed, with $f=\Phi_6^r$ and $g(t)=t^6-1$, we have $\gcd(f,g)=\Phi_6$, implying $F=\Phi_6^{r-1}$, and $G=\Phi_1\Phi_2\Phi_3$. The only relevant primes are $2$, at which the factor $\Phi_3$ of $G$ gives $2$ partial multiplicities equal to $r-1$, and $3$, at which the factor $\Phi_2$ yields $1$ partial multiplicity also equal to $r-1$.
\end{example}

Another interesting family that we can analyze is that of cable knots. Let $K$ be the $(r,s)$-cable of the torus knot $T(p,q)$, with $r,s,p,q\geq2$, $\gcd(r,s)=\gcd(p,q)=1$. Its Alexander polynomial is $\Delta_K(t)=\Delta_{T(r,s)}(t)\Delta_{T(p,q)}(t^r)$ \cite[Eqn. (1)]{Hedden05}. If $s=rpq\pm1$, then $K$ has tunnel number $1$ \cite{MS91}, so its Alexander module is cyclic and Theorem \ref{thm:global} is applicable. For arbitrary cable knots, by \cite[p.~220]{LM85} and having set $h:=\gcd(n,r)$, we have
\begin{equation}\label{eq:cable}
H_1(\Sigma_n(K))\cong H_1(\Sigma_n(T(r,s))) \oplus \left(H_1(\Sigma_{n/h}(T(p,q)))\right)^{ h}.
\end{equation}

\begin{example}
We can make \eqref{eq:cable} even more explicit by invoking \cite[Corollary D]{NoferiniWilliamsBrieskorn}. For instance, for the $(2,8m+5)$-cable $K_m$ of $T(2,2m+1)$, $m\geq1$, also described as $T(5,2;4,m)$ in \cite[Theorem 1]{Lee12}, we obtain
\[H_1(\Sigma_4(K_m))\cong \Z/(2m+1)\Z\oplus\Z/((2m+1)(8m+5))\Z.\]
The covers $\Sigma_4(K_m)$ are thus pairwise not homeomorphic. Moreover, again by \cite[Corollary D]{NoferiniWilliamsBrieskorn}, they cannot be Brieskorn manifolds $\Sigma_w(T(u,v))$ with coprime $u,v$, because their first homology is finite but the nontrivial invariant factors are not all equal.
\end{example}

More generally, recall \cite[Proposition 2.4]{BZ03} that a Seifert surface for a knot $K$ is a compact, connected, oriented surface in $S^3$ whose boundary is $K$. A Seifert surface is associated with an integer matrix $V$,
called a Seifert matrix \cite[Definition 8.5]{BZ03}, such that the matrix $V^T-tV$ presents the Alexander module of $K$ over $\Lambda$ \cite[Theorem 8.8]{BZ03}. For a connected sum of two knots $K_1,K_2$ (denoted $K_1 \# K_2$) each associated with Seifert matrices $V_1,V_2$ we may assume that the Seifert matrix has  the form $V=V_1\oplus V_2$
\cite[Proposition 8.14]{BZ03}. Thus $V^T-tV=(V_1^T-tV_1)\oplus(V_2^T-tV_2)$, and the Alexander module of the connected sum  is the direct sum of those associated with the two knots. Quotienting by $t^n-1$, one similarly sees that the first homology of $\Sigma_n(K_1\mathbin{\#} K_2)$ is isomorphic to the direct sum of those of $\Sigma_n(K_1)$ and  $\Sigma_n(K_2)$.

\begin{example}
In the same notation as Example \ref{ex:Kr}, the knots $K_1\mathbin{\#}K_1$ and $K_2$ both have Alexander polynomial
$\Phi_6^2$. However, $H_1(\Sigma_6(K_1\mathbin{\#}K_1))\cong\Z^4$, whereas $H_1(\Sigma_6(K_2))\cong\Z^2\oplus\Z/2\Z\oplus\Z/6\Z$.
\end{example}

\subsubsection{Hyperbolic knots}

In this subsection, we consider hyperbolic knots in $S^3$. A knot is called hyperbolic if its complement $S^3\setminus K$ admits a complete Riemannian metric of finite volume and constant sectional curvature $-1$ \cite[Sections 0.2.2 and 5.4]{Purcell20}. Following \cite[Section 1]{Lee15}, we denote by $T(a,b;r,s)$ the knot obtained from the torus knot $T(a,b)$ by adding $s$ full twists on $r$ adjacent strands; $s$ may be negative, depending on the orientation of the twists.  Observe that by  \cite[Fact 1.2]{MSK96} the knots in Proposition \ref{prop:hyperbolic} have tunnel number $1$, and by \cite[Corollary 1.2]{Lee15} they are hyperbolic. (For $m=3$, see also the proof of \cite[Lemma 3.1(2)]{Lee15}.)

\begin{proposition}\label{prop:hyperbolic}
Let $m\geq3$ be an odd integer, and let
\[ K_m:=T(m,4;2,-m),\qquad \epsilon:=\begin{cases}
 1,&\text{if }m\equiv1\bmod4,\\
 -1,&\text{if }m\equiv3\bmod4,
 \end{cases} \qquad k:=\frac{m+\epsilon}{2}.\]
Then $K_m$ has Alexander polynomial
\begin{equation}\label{eq:KmAlexander}
 \Delta_{K_m}(t) =\left(\frac{t^m+1}{t+1}\right)^2 \frac{t^{2k}+1}{t^2+1} =\prod_{\substack{a\mid m\\a>1}}\Phi_{2a}(t)^2 \prod_{\substack{b\mid k\\b>1}}\Phi_{4b}(t).
\end{equation}
Moreover, the knots $K_m$ are pairwise distinct for distinct values of $m$.
\end{proposition}

\begin{proof}
We start by the general formula given in \cite[Section 3]{Morton06}. If $0<\rho<m$ and $\rho q\equiv-1\bmod m$, then, up to a power of $t$,
\[ \Delta_{T(m,q;2,s)}(t) =\frac{1-t}{(1-t^m)(1-t^q)}\underbrace{\left(1-\frac{1-t^{2s}}{1+t}
 (t^{\rho q+1}+t^{(m-\rho)q})-t^{mq+2s}\right)}_{:=H(t)}.\]
For $q=4$, we have $\rho=(m-1)/4$ if $m\equiv1\bmod4$, and $\rho=(3m-1)/4$ otherwise. Substituting these values and $s=-m$ yields
\[ \begin{cases}
 t^m(1+t)H(t)=(1-t^{2m})(1+t^m)(1+t^{m+1}),
    &\text{if }m\equiv1\bmod4,\\
 t^{m-1}(1+t)H(t)=(1-t^{2m})(1+t^m)(1+t^{m-1}),
    &\text{if }m\equiv3\bmod4.
 \end{cases}\]
This proves the fractional expression between equalities in \eqref{eq:KmAlexander}. For the factorization on the right hand side of \eqref{eq:KmAlexander}, note that both $m$ and $k$ are odd and apply standard factorizations of the polynomials $t^m+1$ and $t^{2k}+1$.

 Finally, suppose that two knots in the family were equivalent; then, their Alexander polynomials would be equal up to multiplication by a signed power of $t$. Since, for all $m$, $\Delta_{K_m}(0)=1$, it suffices to observe that the map
\[  \{3,5,7,\,\dots\} \rightarrow \mathbb{N}, \qquad m \mapsto \deg \Delta_{K_m}(t) = \begin{cases} 3m-3 \ &\mathrm{if} \ m \equiv 1 \bmod 4,\\
3m-5 \ &\mathrm{if} \ m \equiv 3 \bmod 4
\end{cases} \]  
is injective. Hence, the knots must be pairwise distinct.
\end{proof}

Some concrete examples of the Alexander polynomials in Proposition \ref{prop:hyperbolic} are
\[\Delta_{K_5}=\Phi_{10}^2 \Phi_{12}, \qquad \Delta_{K_9}=\Phi_6^2\Phi_{18}^2\Phi_{20}, \qquad \Delta_{K_{15}}=\Phi_6^2\Phi_{10}^2\Phi_{28}\Phi_{30}^2.\]
We now compute the homology for the branched covers of the knots in Proposition \ref{prop:hyperbolic}. If $n$ is even and $\gcd(m,n)>1$, then $\Delta_{K_m}$ and $t^n-1$ have common cyclotomic factors of multiplicity $2$ in $\Delta_{K_m}$. Thus Corollary \ref{thm:originale} does not apply. Our first step is to apply Theorem \ref{thm:global} to the squared factor in \eqref{eq:KmAlexander}.

\begin{lemma}\label{lem:Pm2}
Let $m\geq 1$ be odd, $n\geq1$, and define $P_m(t):=(t^m+1)/(t+1)$ and $d:=\gcd(m,n)$.
Then
\begin{equation}\label{eq:Pm2}
 \operatorname{coker}\left(P_m(C_{t^n-1})^2\right) \cong
 \begin{cases}
 (\Z/4\Z)^{d-1},&\text{if}\ n \ \text{is odd},\\
 \Z/m^2\Z,&\text{if}\ n \ \text{is even and} \ d=1,\\
 \Z^{d-1}\oplus(\Z/(n/d)\Z)^{d-2}
       \oplus\Z/\left(nm^2/d^2\right)\Z,
       &\text{if}\ n \ \text{is even and} \ d>1.
 \end{cases}
\end{equation}
\end{lemma}

\begin{proof}
Since $P_m(t)=\prod_{a\mid m,\,a>1}\Phi_{2a}(t)$, Corollary \ref{thm:originale} can be applied directly if $n$ is odd or $d=1$ (or both). For odd $n$, the only positive partial multiplicities are at the prime $2$, namely, the only nonunit invariant factor is $4$ ($d-1$ times). For even $n$ and $d=1$, Corollary \ref{thm:originale} yields
one partial multiplicity equal to $2\nu_\ell(m)$ at each prime $\ell\mid m$. These observations settle the first two cases of \eqref{eq:Pm2}.

Thus, we henceforth assume that $n$ is even and $d>1$. Note first $\gcd(P_m^2,t^n-1)=P_d$, implying that the Smith form of $P_m(C_{t^n-1})^2$ has $d-1$ zero diagonal entries. To obtain the invariant factors, we shall apply Theorem \ref{thm:global} to $F:=P_m^2/P_d$ and $G(t):=(t^n-1)/P_d(t)$. Once again, we separately analyze the prime $2$ and the odd primes.
\begin{itemize}
\item For $p=2$, with $\beta:=\nu_2(n)$, every divisor $a>1$ of $d$ corresponds to the sets $A_{2,a}=\{1\}$ and
$B_{2,a}=\{0\}\cup\{2,\ldots,\beta\}$ in the notation of Theorem \ref{thm:global}. The factor $\Phi_{2a}$ has multiplicity $k(2a)=1$ in $F$. The only positive partial multiplicity arises when considering $0 \in B_{2,a}$,
and is equal to $\beta$, and it appears $\sum_{a\mid d,\,a>1}\varphi(a)=d-1$ times.
\item Fix now an odd prime $p$, and define $v:=\nu_p(m)$ and $\beta:=\nu_p(n)$. Applying again Theorem \ref{thm:global}, positive partial multiplicities can occur only for indices whose coprime-to-$p$ part is $2a$, with $a\mid d/p^{\min(v,\beta)}$. If $\beta\leq v$, the only factor of $G$ with such an index
is $\Phi_2$. The corresponding factors of $F$ are $\Phi_{2p^h}$, $1\leq h\leq v$, with multiplicity $1$ for $h\leq\beta$ and $2$ otherwise. Inputting these data into Theorem \ref{thm:global} we get  $1$ partial multiplicity equal to $\beta+2(v-\beta)=2v-\beta$. If instead $\beta>v\geq1$, the factors $\Phi_{2ap^h}$ of $F$ have multiplicities $k(2ap^h)=1$ for every $h\in A_{p,2a}$. For each block $D_{p,2a}$ in Theorem \ref{thm:global}, the division in \eqref{eq:division} therefore gives $q_{v+1}=0$ and the following table:
\[ \begin{array}{c|c|c|c}
 &A_{p,2a}&B_{p,2a}&r_{v+1}\\ \hline
 a=1&\{1,\ldots,v\}&\{0\}\cup\{v+1,\ldots,\beta\}&p^v-2\\
 a>1&\{0,\ldots,v\}&\{v+1,\ldots,\beta\}&p^v
 \end{array}\]
The resulting positive partial multiplicities are, for each value of $a$, $\beta-v$ (repeated $r_{v+1}$ times); in addition, for the case $a=1$ only, there is also one partial multiplicity equal to $\beta$. Taking into account that
$\sum_{a\mid d/p^v}\varphi(a)=d/p^v$, overall we obtain $d-2$ partial multiplicities equal to $\beta-v$ and $1$ partial multiplicity equal to $\beta$. Finally, if $v=0<\beta$, when $a=1$ there is no corresponding factor in $F$; and for each $a>1$, $A_{p,2a}=\{0\}$ and $B_{p,2a}=\{1,\ldots,\beta\}$ yield $1$ partial multiplicity equal to $\beta$, and thus there are overall $d-1$ partial multiplicities equal to $\beta$ after summing over $a$.
\end{itemize}
At every prime $p$, the positive partial multiplicities obtained above are precisely the positive entries of the list consisting of $d-2$ copies of $\nu_p(n/d)$ and $1$ copy of $\nu_p(nm^2/d^2)$. This concludes the proof of \eqref{eq:Pm2}.
\end{proof}

\begin{corollary}\label{cor:Km}
Let $K_m$ and $k$ be as in Proposition \ref{prop:hyperbolic}. For $n\geq1$, define $d:=\gcd(m,n)$ and $e:=\gcd(k,n)$. Then $H_1(\Sigma_n(K_m))$ is isomorphic to
\begin{equation}\label{eq:hyperbolic-all-covers}
 \begin{cases}
 (\Z/2\Z)^{e-1}\oplus(\Z/4\Z)^{d-1},
       &\text{if} \ n \ \text{is odd},\\[4pt]
 (\Z/2\Z)^{2(e-1)}\oplus\Z/m^2\Z,
       &\text{if}\ n  \equiv2\bmod4,\ d=1,\\[4pt]
 \Z^{d-1}\oplus(\Z/2\Z)^{2(e-1)}\oplus(\Z/(n/d)\Z)^{d-2}
             \oplus\Z/\left(nm^2/d^2\right)\Z,
       &\text{if} \ n\equiv2\bmod4,\ d>1,\\[4pt]
 \Z^{2(e-1)}\oplus\Z/(k/e)\Z\oplus\Z/(m^2k/e)\Z,
       &\text{if} \ n\equiv0\bmod4,\ d=1,\\[4pt]
 \Z^{d+2e-3}\oplus(\Z/(k/e)\Z)^2\oplus(\Z/(n/d)\Z)^{d-2}
             \oplus\Z/\left(nm^2/d^2\right)\Z,
       &\text{if} \ n\equiv0\bmod4,\ d>1.
 \end{cases}
\end{equation}
\end{corollary}

\begin{proof}
Let $Q_k(t):=P_k(t^2)$. Since $\gcd(m,k)=1$, we have $\lvert\operatorname{Res}(P_m^2,Q_k)\rvert=1$ by the classic resultant formulae in \cite{Apostol70}. By the Swap Theorem, Theorem \ref{lem:bello}, and
\eqref{eq:cchomology}, we see that $H_1(\Sigma_n(K_m))$ is isomorphic to $\operatorname{coker} P_m(C_{t^n-1})^2 \oplus \operatorname{coker} Q_k(C_{t^n-1})$. Lemma \ref{lem:Pm2} yields the first cokernel; for the second, when $n$ is odd Corollary \ref{thm:originale} yields $(\Z/2\Z)^{e-1}$. If $n$ is even, we invoke the
Composition Theorem \cite[Theorem 3.13]{NoferiniWilliamsApplications} to obtain 
\[\operatorname{coker} Q_k(C_{t^n-1}) \cong (\operatorname{coker}P_k(C_{t^{n/2}-1}))^2.\]
Observing that $\gcd(k,n/2)=e$ and $\gcd(m,k/e)=1$, we obtain the formulae for the $4$ subcases corresponding to even $n$ in \eqref{eq:hyperbolic-all-covers} by applying Corollary \ref{thm:originale} and combining the result with Lemma \ref{lem:Pm2}.
\end{proof}

Note that the isomorphism described by \eqref{eq:hyperbolic-all-covers} does not generally yield the invariant factors of the homology group, because the divisibility property is not necessarily satisfied. One can nevertheless easily recover the invariant factors using Lemma \ref{lem:reordersmith}.

\begin{example}
For $n=2$, Corollary \ref{cor:Km} gives $H_1(\Sigma_2(K_m))\cong\Z/m^2\Z$. This specialization coincides with the analogous specialization of \cite[Table 1 and Theorem 4.17]{DRS15}. For $n=2m$, we obtain instead
\[ H_1(\Sigma_{2m}(K_m)) \cong\Z^{m-1}\oplus(\Z/2\Z)^{m-2}\oplus\Z/(2m)\Z.\]
For this family, Corollary \ref{cor:Km} yields the full first homology, thus generalizing \cite[Theorem 3.4]{SilverWilliamsTorsion} that determines the orders of the torsion subgroups.
 
 For another example, if we set $m=15$ and $n=90$ then $\Delta_{K_{15}}$ and $t^{90}-1$ have the $3$ common cyclotomic factors $\Phi_6$, $\Phi_{10}$, and $\Phi_{30}$, each of multiplicity $2$ in $\Delta_{K_{15}}$, and
\[H_1(\Sigma_{90}(K_{15}))\cong\Z^{14}\oplus(\Z/6\Z)^{13}\oplus\Z/90\Z.\]
\end{example}

Another family of hyperbolic knots that our results allow us to analyze are the  pretzel knots $P(2,m,-m)$, where the parameters count signed half-twists in the three twist regions \cite[Introduction]{Belousov2026}. Moreover, the tunnel number of these knots is $1$ \cite[Section 5]{GMM05}.
\begin{corollary}\label{cor:pretzel}
Let $m\geq3$ be odd, $n\geq1$, and $d:=\gcd(m,n)$. The pretzel knot $J_m:=P(2,m,-m)$ is hyperbolic and has Alexander polynomial $\left(\frac{t^m+1}{t+1}\right)^2$. Moreover, $H_1(\Sigma_n(J_m))$ is isomorphic to
\begin{equation}\label{eq:pretzel}
 \begin{cases}
 (\Z/4\Z)^{d-1},&\text{if} \ n \ \text{is odd},\\
 \Z/m^2\Z,&\text{if}\ n \ \text{is even and} \ d=1,\\
 \Z^{d-1}\oplus(\Z/(n/d)\Z)^{d-2}
       \oplus\Z/\left(nm^2/d^2\right)\Z,
       &\text{if} \ n \ \text{is even and} \ d>1.
 \end{cases}
\end{equation}
\end{corollary}

\begin{proof}
Specializing \cite[Theorem 1(1)]{Belousov2026} to the triple $(2,m,-m)$ we obtain $\Delta_{J_m}=P_m^2$ where $P_m$ is defined as in Lemma \ref{lem:Pm2}. Moreover, by \cite[Remark 1]{IJ09},  $J_m$ is either a hyperbolic knot, a torus knot, or the unknot. However, the last two possibilities are impossible, because $P_m^2$ is nonconstant and not squarefree. Finally, \eqref{eq:pretzel} immediately follows from \eqref{eq:cchomology} and Lemma \ref{lem:Pm2}.
\end{proof}

For $m=3$, the homology \eqref{eq:pretzel} can also be obtained by (a special case of) Mednykh's arguments, given in \cite[Sections 2--4 and 6]{Med21}. Even though those arguments are stated for two-bridge knots, they are still applicable because the Alexander module is cyclic and $\Delta_{J_3}=\Phi_6^2$.

\subsubsection{Knots in lens spaces}\label{sec:kls}

Following \cite{Cattabriga06}, we now consider $(1,1)$-knots in $3$-dimensional lens spaces, which we denote by $L(P,Q)$; they are a parametric family of $3$-dimensional closed manifolds that includes the sphere as the special case $L(1,0)=S^3$. A cyclic branched cover of degree $n$ is called strongly cyclic if its branching index along the knot is $n$ (see \cite[Section 2]{Cattabriga06} for more details). For a $(1,1)$-knot in $L(P,Q)$, such a cover exists and is unique up to equivalence whenever $\gcd(n,P)=1$; see the same section.

Let $K\subset L(P,Q)$ be a $(1,1)$ knot, $P\geq1$, and let $d$ be the order of the torsion subgroup of $H_1(L(P,Q)\setminus K)$. Suppose that its Alexander polynomial, up to signed powers of $t$, is a product of cyclotomic polynomials $\Delta_K$. For $n>1$ coprime to $P$, let $\Sigma_n(K)$ be the $n$-fold strongly cyclic branched cover. By the proofs of \cite[Theorem 4.1 and Corollary 4.1 ]{Cattabriga06}, we have that $d \mid P$ and
\[ H_1(\Sigma_n(K))\cong \operatorname{coker}f(C_{t^n-1}),\qquad f(t):=\Delta_K(t)\frac{t^{P/d}-1}{t-1}. \]
Thus $f$ is a product of cyclotomic polynomials.

For instance, Cattabriga \cite[Example 4.2]{Cattabriga06} studies the $(1,1)$-knots $K_m=K(1,m-2,0,1)\subset L(m-2,1)$. For odd $m\geq3$, for these knots $d=1$ and
\[ \Delta_{K_m}(t)=\frac{t^m+1}{t+1},\qquad f(t)=\frac{(t^m+1)(t^{m-2}-1)}{t^2-1}.\]
Corollary \ref{cor:typeZodd} gives the homology of $\Sigma_n(K_m)$ for every $n>1$ coprime to $m-2$.

\begin{corollary}\label{cor:cattacovers}
Let $m\geq3$ be odd and $K_m=K(1,m-2,0,1)\subset L(m-2,1)$. For every $n>1$ such that $\gcd(n,m-2)=1$, let $\Sigma_n(K_m)$ be its $n$-fold strongly cyclic branched cover, and let $h:=\gcd(m,n)$. Then
\begin{equation}\label{eq:cattacovers}
H_1(\Sigma_n(K_m))\cong
\begin{cases}
\Z/(m-2)\Z, &\text{if} \ n \ \text{is odd and} \ h=1,\\
(\Z/2\Z)^{h-2}\oplus\Z/(2(m-2))\Z,
 &\text{if} \ n \ \text{is odd and} \ h>1,\\
\Z^{h-1}\oplus\Z/(m(m-2)/h)\Z, &\text{if} \ n \ \text{is even}.
\end{cases}
\end{equation}
\end{corollary}

\begin{proof}
We apply Corollary \ref{cor:typeZodd} with $c=m$ and $d=m-2$. Since $\gcd(n,m-2)=1$ and $\gcd(m,m-2)=1$, we obtain
$(\Z/2\Z)^{h-1}\oplus\Z/(m-2)\Z$ when $n$ is odd, and $\Z^{h-1}\oplus\Z/(m/h)\Z\oplus\Z/(m-2)\Z$ when $n$ is even.
To obtain the case of $n$ odd and $h>1$ in \eqref{eq:cattacovers}, it suffices to note that $m-2$ is odd and combine one copy of $\Z/2 \Z$ with $\Z/(m-2) \Z$. A similar trick yields the case of even $n$, noting also that $m-2$ is coprime to $m/h$.
\end{proof}

\begin{example}
Two interesting special cases of Corollary \ref{cor:cattacovers} are $n=m$ and $n=2m$, for which respectively
\[ H_1(\Sigma_m(K_m))\cong(\Z/2\Z)^{m-2}\oplus\Z/(2(m-2))\Z \quad \mathrm{and} \quad H_1(\Sigma_{2m}(K_m))
\cong\Z^{m-1}\oplus\Z/(m-2)\Z.\]
In particular, $\Sigma_m(K_m)$ is the Neuwirth manifold considered in \cite[Example 4.2]{Cattabriga06}. Its homology is also the specialization $(\alpha,\beta,\ell)=(2,1,1)$ of \cite[Corollary 4.6]{NoferiniWilliamsApplications}.
\end{example}

\subsubsection{Graphs associated with circulant matrices}

Letting $C:=C_{t^n-1}$ and $n,v\geq1$, the matrices in the family $H_{n,v}:=(I+C^{v+1})(I+C+\cdots+C^{v-1})$ are circulant matrices related to adjacency matrices of certain graphs. Williams \cite[Sections 5--6]{GW14} determines their  Smith form in several cases. Applying Corollary \ref{cor:typeZeven} with $a=v+1$ and $b=v$, we obtain the following formula that covers the whole family.

\begin{corollary}\label{cor:cyclepowers}
Let $n,v\geq1$, and define $d:=\gcd(n,v+1)$ and $e:=\gcd(n,v)$. Then, defining $H_{n,v}$ as above,
\[ \operatorname{coker}H_{n,v}\cong
\begin{cases}
\Z^{e-1}\oplus(\Z/2\Z)^{d-1}\oplus\Z/(2v/e)\Z,
 &\text{if} \ \nu_2(n)\leq\nu_2(v+1),\\
\Z^{d+e-1}\oplus\Z/(v/e)\Z,
 &\text{if} \ v \ \text{is odd and} \ \nu_2(n)>\nu_2(v+1),\\
\Z^{d+e-2}\oplus\Z/(2nv/(de^2))\Z,
 &\text{if} \ n,v \ \text{are even}.
\end{cases}\]
\end{corollary}

\begin{proof}
In the notation of Corollary \ref{cor:typeZeven}, we have $h=1$ and $L=n/(de)$ because $\gcd(v,v+1)=1$. The statement follows.
\end{proof}

Recall that the $v$-th power of a cycle graph joins $2$ distinct vertices when their distance along the cycle is at most $v$.

\begin{example}
When $1 \leq v < n/2$, $C^{-v}H_{n,v}$ is the adjacency matrix of the $v$-th power of the cycle graph with $n$ vertices \cite[Section 2]{GW14}. Since $C^{-v}$ is unimodular, the Smith form of the latter adjacency matrix is the same as  that of $H_{n,v}$.

For example, when $n=20$ and $v=4$, we obtain from Corollary \ref{cor:cyclepowers} (and the previous observation) that the adjacency matrix of the $4$-th power of the cycle with $20$ vertices has Smith form
$I_{12}\oplus2\oplus0_{7\times7}$.
\end{example}

For $n\geq5$, let us now consider the graph $G_n$ obtained as the Cartesian product of the complete graph with two vertices and the square (i.e., $2$nd power) of the cycle with $n$ vertices. The graph $G_n$ is regular of degree $5$, and hence its Laplacian is $L_n:=5I_{2n}-A_n$, where $A_n$ is the adjacency matrix of $G_n$. In \cite[Section 5.3]{KMM25}, Kwon, Mednykh and Mednykh observe that $\det(L_n-6I_{2n})$ is periodic in $n$, with period $6$. Corollary \ref{cor:graphL} goes beyond this observation by giving the full list of the invariant factors.

\begin{corollary}\label{cor:graphL}
Let $n\geq5$, and let $L_n$ be the graph Laplacian of the graph $G_n$ defined as above. Then, the Smith form of $L_n-6 I_{2n}$ is
\[I_{2n-r_n-1}\oplus d_n\oplus0_{r_n\times r_n},\]
where $r_n:=\gcd(n,4)+\gcd(n,6)-\gcd(n,2)-1$, and $d_n:=24/\gcd(n,3)$ if $n$ is odd, while $d_n:=24n/\gcd(n,12)^2$ if $n$ is even.
\end{corollary}
\begin{proof}
Let $C:=C_{t^n-1}$ and $S:=I_n+C+C^{-1}+C^2+C^{-2}$. With no loss of generality up to reordering vertices (which corresponds to a permutation similarity and hence an equivalence), we have $6I_{2n}-L_n=\begin{bmatrix}S&I_n\\
I_n&S\end{bmatrix}$.
Having defined
\[ f(t):=(1+t+t^2+t^3)(1+t+t^2+t^3+t^4+t^5) =\Phi_2(t)^2\Phi_3(t)\Phi_4(t)\Phi_6(t),\] 
we then have
\[\begin{bmatrix}I_n&0\\-S&I_n\end{bmatrix}(6I_{2n}-L_n)\begin{bmatrix}0&-I_n\\I_n&S\end{bmatrix}
 = I_n\oplus(S^2-I_n)\sim I_n\oplus f(C).\]
Indeed, $t^4((1+t+t^{-1}+t^2+t^{-2})^2-1)=f(t)$ and hence $S^2-I_n=C^{-4}f(C)$. Moreover, $C$ is a permutation matrix and thus unimodular, and therefore $C^{-4}$ also is.

The degree of $z:=\gcd(f,t^n-1)$ is $r_n$. Applying Theorem \ref{thm:global} to $F:=f/z$ and $G(t):=(t^n-1)/z(t)$ shows that, for each prime $p$, $n-r_n-1$ partial multiplicities of $F(C_G)$ are $0$. Writing $\beta:=\nu_2(n)$ and $\gamma:=\nu_3(n)$, the last partial multiplicity is $3+\beta-2\min\{\beta,2\}$ at $2$; at $3$, it is $\max\{1-\gamma,0\}$ if $n$ is odd and $|1-\gamma|$ if $n$ is even; for $p>3$, it is $0$ if $n$ is odd and $\nu_p(n)$ if $n$ is even. These are precisely the valuations of $d_n$, proving the statement.
\end{proof}

For example, for every $k\geq1$ we obtain $\operatorname{coker}(L_{12k}-6I_{24k})\cong
\Z^7\oplus\Z/(2k)\Z$. Thus the determinants are periodic, although the torsion subgroups of these cokernels have unbounded order.

\subsection{Equivariant \texorpdfstring{$K$}{K}-theory}\label{sec:Ktheory}

Let $N\geq2$, $r\geq1$, and $1\leq a_j<N$ for $j=1,\ldots,r$. Letting $\zeta:=e^{2\pi i/N}$, we consider the following map on $\C^r$:
\[(z_1,\ldots,z_r)\mapsto (\zeta^{a_1}z_1,\ldots,\zeta^{a_r}z_r)\]
It is an automorphism of the unit sphere $S(\C^r)$, and applying it $N$ times gives the identity. Its powers therefore define an action of the cyclic group $C_N:=\Z/(N\Z)$ on $S(\C^r)$. For each $j$, let $\C_{a_j}$ denote the $j$-th copy of $\C$, on which the chosen generator of $C_N$ acts by $z\mapsto\zeta^{a_j}z$. In this notation, $\C^r=\bigoplus_{j=1}^r\mathbb C_{a_j}$.

Following \cite[Sections 1--2]{Segal68}, we recall the definition of the associated $K$-group in terms of equivariant complex vector bundles over $S(\C^r)$. Such a bundle attaches a finite-dimensional complex vector space to each point of the sphere, and locally has the form $U\times\mathbb C^k$, where $U$ is an open subset of the sphere. We require the action of $C_N$ to extend continuously to the bundle, mapping linearly the vector space over each point onto the vector space over the transformed point. The abelian group $K^0_{C_N}(S(\C^r))$ has a generator $[E]$ for each isomorphism class of such bundles, with the relations $[E\oplus F]=[E]+[F]$. The connection with our results comes from the polynomial presentation
\begin{equation}\label{eq:equivariant-k}
K^0_{C_N}(S(\C^r))\cong\Z[t]/\langle f,g\rangle,\qquad f(t):=\prod_{j=1}^r(t^{a_j}-1),\quad g(t):=t^N-1,
\end{equation}
which follows from \cite[pp. 136-141]{Segal68} and \cite[Lemma 1.3.3]{brunerbook}.
Since $f$ is a product of cyclotomic polynomials and $g$ is squarefree, Theorem \ref{thm:global} yields both the free subgroup and the invariant factors of the torsion subgroup of $K^0_{C_N}(S(\C^r))$. This also applies when the action is not free, i.e., there exists a group element other than the identity that fixes a point of $S(\C^r)$. We now consider $N=p^3$, where $p$ is prime, with $u$ of the integers $a_j$ equal to $1$ and the remaining $v=r-u$ equal to $p$. For $u,v\geq1$, division by the gcd gives $G=g/\gcd(f,g)=\Phi_{p^2}\Phi_{p^3}$. Varying $u$ and $v$ then illustrates both cases of \eqref{eq:closed}.

\begin{corollary}\label{cor:illustKtheory}
Let $p$ be prime, $N=p^3$, and $u,v\geq1$. Consider the action of $C_{p^3}$ on $\C^{u+v}$ generated by the matrix
$\zeta I_u\oplus\zeta^p I_v$, where $\zeta:=e^{2\pi i/p^3}$. Define $D:=u+pv-p$, $e_2:=\varphi(p^2)=p(p-1)$, and
$e_3:=\varphi(p^3)=p^2(p-1)$. If $D<e_2$, then
\[K^0_{C_{p^3}}(S(\C^{u+v}))\cong\Z^p\oplus(\Z/p^2\Z)^D.\]
If $D\geq e_2$, write the Euclidean divisions $D=:q_2e_2+r_2$, $0\leq r_2<e_2$, and $D-e_2=:q_3e_3+r_3$, $0\leq r_3<e_3$. Then,
\[K^0_{C_{p^3}}(S(\C^{u+v}))\cong\Z^p\oplus(\Z/p^{q_3}\Z)^{e_3-r_3}\oplus(\Z/p^{q_3+1}\Z)^{r_3}\oplus(\Z/p^{q_2+1}\Z)^{e_2-r_2}\oplus(\Z/p^{q_2+2}\Z)^{r_2}.\]
\end{corollary}

\begin{proof}
In \eqref{eq:equivariant-k}, we have $f(t)=(t-1)^u(t^p-1)^v=\Phi_1(t)^{u+v}\Phi_p(t)^v$ and $g(t)=t^{p^3}-1$. Hence, $\gcd(f,g)=\Phi_1\Phi_p$, which has degree $p$; this yields the free subgroup  $\Z^p$. Dividing $f$ and $g$ by this gcd, we obtain $F(t)=\Phi_1(t)^{u+v-1}\Phi_p(t)^{v-1}$ and $G(t)=\Phi_{p^2}(t)\Phi_{p^3}(t)$. In the notation of Theorem \ref{thm:global}, the only relevant prime is $p$, the set $B_{p,1}=\{2,3\}$, and all cyclotomic factors of $F$ have indices less than $p^2$. Since $\deg F=D$, the left hand side of \eqref{eq:division} equals $D$ for $b=2$ and $\max(D-e_2,0)$ for $b=3$. Moreover, $\alpha_2=\delta_2=1$ and $\alpha_3=\delta_3=0$. If $D<e_2$, then $(q_2,r_2)=(0,D)$ and $(q_3,r_3)=(0,0)$, so \eqref{eq:closed} gives, over $\Z_p$,
\[ F(C_G)\sim I_{e_2-D}\oplus p^2I_D\oplus I_{e_3}.\]
Therefore $\operatorname{coker} F(C_G) \cong (\Z/p^2\Z)^D$. Suppose now that $D\geq e_2$. We have $q_2\geq1$, so for $b=2$ we use the second case of \eqref{eq:closed}, with $\alpha_2=1$. For $b=3$, since $\alpha_3=\delta_3=0$, both cases
of \eqref{eq:closed} give the same expression. Thus, over $\Z_p$,
\[F(C_G)\sim p^{q_3}I_{e_3-r_3}\oplus p^{q_3+1}I_{r_3} \oplus p^{q_2+1}I_{e_2-r_2}\oplus p^{q_2+2}I_{r_2}.\]
 This concludes the proof.
\end{proof}
\begin{example}
Let $p=u=v=2$. Then,
\[K^0_{C_8}(S(\C^4))\cong \Z^2\oplus(\Z/2\Z)^2\oplus(\Z/8\Z)^2.\]

When $(u,v)\in \{(2,1),(1,2),(3,1)\}$ and for an arbitrary prime $p$, the groups in Corollary \ref{cor:illustKtheory} agree with (some special cases of) the $K_0$-groups of quantum weighted lens spaces computed in \cite[Examples 2.4--2.6]{BS18}, on setting $k=p^2$ and $l=p$ in their notation. We note also that the corrigendum given in \cite[Remark 3.5]{GZ24} does not affect the three examples.
\end{example}

\begin{example}\label{ex:lensgroups}
If every $a_j=1$ and $r=d+1$, then the quotient $S(\C^r)/C_N$ is the $(2d+1)$-dimensional standard lens space\footnote{This nomenclature and the notation are adapted from \cite{Yanagida76}, and differ from the ones in \cite{Cattabriga06}; in the notation used here, the $3$-dimensional lens spaces discussed in Section \ref{sec:kls} would be denoted as $L(P;1,Q)$. The more compact $\mathcal{L}(d,N)$ is our own introduction, justified both for brevity and to avoid confusion.}
\[ \mathcal{L}(d,N):=L(N;\underbrace{1,\dots,1}_{d+1 \ \mathrm{times}}).   \]
 The action of $C_N$ on the sphere is free, and \cite[Proposition 2.1]{Segal68} identifies its ordinary $K$-group with the equivariant group in \eqref{eq:equivariant-k}. Moreover, it holds $K^0(\mathcal{L}(d,N)) \cong \Z \oplus \widetilde K^0(\mathcal{L}(d,N))$, where the reduced $K$-group admits the presentation (see also \cite[Lemma 2.7]{Hashimoto83})
\[ \widetilde K^0(\mathcal{L}(d,N)) \cong\Z[t]/\langle 1+t+\cdots+t^{N-1},(t-1)^d\rangle.\]
By \cite[Proposition 3.4.3(1)]{brunerbook}, the latter quotient is in turn isomorphic to the connective $K$-homology group $ku_{2d-1}(BC_N)$, where $BC_N$ denotes the classifying space of $C_N$.

In the world of polynomials and Smith forms, moving from the ordinary to the reduced $K$-group  corresponds to dividing both $f(t)=(t-1)^{d+1}$ and $g(t)=t^N-1$ by their gcd, namely $t-1$.
In this setting and when $N=p^a$, Theorem \ref{thm:global} recovers \cite[Theorem 6.1]{Yanagida76}.
\end{example}

\begin{remark}\label{rem:coker}
    Recall from the Introduction that the matrix $f(C_g)$ represents, in the monomial basis, multiplication by $[f]$ in the quotient ring $Q_g:=\Z[t]/\langle g \rangle$. Thus, we have the $\Z$-module isomorphism
    \begin{equation}\label{eq:isomorphism}
        \operatorname{coker} f(C_g) := Q_g/\langle f \rangle \cong \Z[t]/\langle f,g \rangle.   
    \end{equation} 
   A fortiori, \eqref{eq:isomorphism} can be interpreted as an isomorphism of additive groups. Merging this observation with \eqref{eq:equivariant-k} and Example \ref{ex:lensgroups}, we conclude
   \[ \Z \oplus \widetilde K^0(\mathcal{L}(d,N)) \cong  K^0(\mathcal{L}(d,N)) \cong \operatorname{coker} (C_{t^N-1}-I)^{d+1}  \]
\end{remark}

\subsection{Smith forms of matrix powers}\label{sec:matrixpowers}

Motivated by a question posed by Bruner, in \cite[Theorem 1.1]{Nof26} Noferini  proves that the Smith forms of powers of square integer matrices behave in an eventually periodic manner. To illustrate a connection with the lens space groups discussed in Example \ref{ex:lensgroups}, let $p$ be prime and $a\geq1$, and define $A:=C_{t^{p^a}-1}-I$. By Remark \ref{rem:coker}, the cokernel of the $(d+1)$-th power of $A$ is, up to isomorphism,  the $K$-group of the lens space $\mathcal{L}(d,p^a)$. Moreover, letting $m:=d+1$, \cite[Theorem 6.1]{Yanagida76} yields, for $m\geq p^{a-1}$, that the Smith form of $A^{m+\varphi(p^a)}$ is equal to a diagonal matrix $D$ times the Smith form of $A^m$, where
\[D:=\left(\bigoplus_{j=0}^{a-1}p^{p^j}I_{\varphi(p^{a-j})}\right)\oplus1.\]
(In \cite{Yanagida76}, it is assumed $a>1$, but for $a=1$ the same conclusion follows more directly from \eqref{eq:example1}.) 

Theorem \ref{thm:final} below shows that, under appropriate assumptions, the least eventual period is determined by the eigenvalues. Its statement does not generalize to powers of arbitrary square integer matrices. For example, given a prime $p$, the eigenvalues of $\begin{bmatrix}p&1\\0&p\end{bmatrix}$ have $p$-adic valuation $1 \in \Z$, but the least eventual period is $p \neq 1$ \cite[Section 1]{Nof26}. 

\begin{theorem}\label{thm:final}
Let $f,g\in\Z[t]$ be products of cyclotomic polynomials, with $g$ squarefree. Define $X:=f(C_g)$, $n:=\deg g$, and $r:=\operatorname{rank}X$. For every prime $p$, denote the $p$-adic valuations of the nonzero eigenvalues of $X$, counted with multiplicity, by $\gamma_{p,1}\leq\cdots\leq\gamma_{p,r}$.
Let $T$ be the least common denominator of $ \{\gamma_{p,i}\}$, taken over all $i=1,\dots,r$ and all primes $p$ dividing $\operatorname{Res}(f,\frac{g}{\gcd(f,g)})$. Then, for all sufficiently large values of $m$,
\[\SF(X^{m+T})=D\SF(X^m),\qquad D:=\left(\bigoplus_{i=1}^r\prod_p p^{T\gamma_{p,i}}\right) \oplus I_{n-r}.\]
Moreover, $T$ is the least possible such eventual period.
\end{theorem}

\begin{proof}
Let $z:=\gcd(f,g)$, $G:=g/z$, and $F_m:=f^m/z$. Since $g$ is squarefree, $\gcd(f^m,g)=z$ for every
$m\geq1$. Hence, by \cite[Theorem A]{NoferiniWilliamsBrieskorn}, $X^m\sim F_m(C_G)\oplus0_{(n-r)\times(n-r)}$ and $\deg G=r$. Since $F_m$ divides $f^m$, only the partial multiplicities of the primes dividing $\operatorname{Res}(f,G)$ can be positive.

Fix one such prime $p$, and let $\eta_{p,1}(m)\leq\cdots\leq\eta_{p,r}(m)$ be the finite partial multiplicities of $X^m$ over $\Z_p$. By \cite[Section 2.2 and proof of Theorem 2.5]{Nof26}, $\eta_{p,i}(m)/m\to\gamma_{p,i}$ as $m\to\infty$.

For every cyclotomic factor $\Phi_d$ of $f$, denote its multiplicity in $f$ by $u(d)$. For $m\geq2$, its multiplicity in $F_m$ is thus $mu(d)-1$ if $\Phi_d\mid g$, or $mu(d)$ otherwise. Hence, when applying Theorem \ref{thm:global} to $f^m$ and $g$, the index sets and $\delta_b$ do not depend on $m$, while $\alpha_b(m) \in \Z[m]$ is an affine function of $m$.  Denote now $e_b:=\varphi(p^b)$. For each pair $(a,b)$, the left hand side of \eqref{eq:division} has the form $N_b(m):=\max(ms-c,0)$ for constant nonnegative integers $s,c$. There are two cases: (i) If $s=0$, the only exponent in \eqref{eq:closed} is $\alpha_b(m)-\delta_b$; (ii) If $s>0$,  then $q_b(m)\geq1$ for large enough $m$ and the exponents in \eqref{eq:closed} are
\[ \alpha_b(m)+ \left\lfloor\frac{N_b(m)+j}{e_b}\right\rfloor, \qquad j=0,\ldots,e_b-1.\]
In both cases, these are expressions of the form $\lfloor m\lambda+\beta\rfloor$ where $\lambda,\beta \in \QQ$.

By \eqref{eq:globallocal}, the finite partial multiplicities of $X^m$ over $\Z_p$ are obtained by collecting these exponents for all pairs $(a,b)$, each repeated $\varphi(a)$ times, and reordering the resulting list. For large enough $m$, this ordering is determined by the values of the slopes $\lambda$, and in case of tie by the value of the constants $\beta$. Since $\lfloor m\lambda+\beta\rfloor/m\to\lambda$ and $\eta_{p,i}(m)/m\to\gamma_{p,i}$ \cite{Nof26}, the expression for the $i$-th exponent must satisfy $\lambda=\gamma_{p,i}$. Thus, there exist constants $\beta_{p,i}$ such that
$\eta_{p,i}(m)=\lfloor m\gamma_{p,i}+\beta_{p,i}\rfloor$ for all large enough $m$. Since $T\gamma_{p,i}\in\Z$, it follows that $\eta_{p,i}(m+T)=\eta_{p,i}(m)+T\gamma_{p,i}$. After recovering the last $n-r$ zero diagonal entries, Lemma \ref{lem:reordersmith} yields the sought relation between $\SF(X^{m+T})$ and $\SF(X^m)$.

Finally, if $L$ is any eventual period with corresponding diagonal matrix $E$, the same argument on asymptotic growth gives $\nu_p(E_{ii})=L\gamma_{p,i}\in\Z$, for all $p$. Hence $T\mid L$.
\end{proof}

\begin{example}
This example is due to Bruner \cite{email}. Let $g=\Phi_1\Phi_4$ and let $A$ be the companion matrix of $g(t-1)=t^3-4 t^2+6t -4$. By Corollary \ref{cor:composition}, $A^m$ is unimodularly equivalent to $f(C_g)$ with $f=\Phi_2^m$, for all $m$. Note that $A$ has eigenvalues $2$ and $1 \pm i$, with $\nu_2(2)=1$ and, observing $(1\pm i)^2 = \pm 2i$, $\nu_2(1 \pm i)=\frac12$. Thus, by Theorem \ref{thm:final} the least eventual period is $2$; indeed, one can check via Theorem \ref{thm:closed} that $\SF(A^{m+2})=(2I_2\oplus4)\SF(A^m)$ for all $m\geq1$. 
\end{example}

\begin{example}
Let $c\geq4$, $u,v\geq1$, and define
\[ G_c(t):=\Phi_1(t)\Phi_4(t)\Phi_{2^c}(t),\qquad X:=\Phi_2(C_{G_c})^u\Phi_8(C_{G_c})^v
 \in\Z^{(2^{c-1}+3)\times(2^{c-1}+3)}.\]
For every $m\geq1$, write $m(u+4v)-3=:2^{c-1}q_m+r_m$, with $0\leq r_m<2^{c-1}$. Since $|\det X|$ is a power of $2$ \cite[Theorem 7.4]{NoferiniWilliamsBrieskorn}, applying Theorem \ref{thm:closed} with $p=2$,
\[ \SF(X^m) = 2^{q_m}I_{2^{c-1}-r_m} \oplus2^{q_m+1}I_{r_m} \oplus2^{mv+1+\lfloor(mu-1)/2\rfloor} \oplus 2^{mv+1+\lfloor mu/2\rfloor} \oplus2^{m(u+v)+2}.\]
Moreover, by Theorem \ref{thm:final}, the least eventual period and the corresponding diagonal matrix are
\[ T:=\frac{2^{c-1}}{\gcd(2^{c-1},u+4v)},\qquad D:=2^{T(u+4v)/2^{c-1}}I_{2^{c-1}}\oplus2^{T(u+2v)/2}I_2\oplus2^{T(u+v)}.\]
Hence, observing that $Tu$ is even, $\SF(X^{m+T})=D\SF(X^m)$ for all $m\geq1$.
\end{example}

\section*{AI statement}

The ideas and proofs in this paper are human. We acknowledge use of ChatGPT 6.0 for literature review and auditing of the proofs and the exposition; some (minor) gaps in the original proofs of Lemma \ref{lem:compablocchi} and Lemma \ref{lem:Pm2} were found and fixed after this audit.

\section*{Acknowledgements}

We thank Robert Bruner who shared with us his observations about the link between K-theory, cyclotomic fields, and Smith forms of powers of companion matrices \cite{email}. In later correspondence about \cite{Nof26}, Bruner also pointed Noferini to \cite{brunerbook}. Williams thanks the Department of Mathematics and Systems Analysis at Aalto University for its hospitality during a research visit in Summer 2026, when part of this work took place.

\end{document}